\documentclass[12pt]{amsart}

\usepackage{amssymb,amsthm,amsmath,amsxtra}
\usepackage{color}
\usepackage[usenames,dvipsnames]{xcolor}
\usepackage[all]{xy}
\usepackage{fullpage}
\usepackage{comment}
\usepackage{enumerate}
\usepackage{bbm}
\usepackage{hyperref}
\hypersetup{colorlinks=true,urlcolor=blue,citecolor=blue,linkcolor=blue}
\usepackage{colonequals} 
\usepackage{mathrsfs}
\usepackage{multirow}
\allowdisplaybreaks 
\numberwithin{equation}{subsection}
\usepackage{mathtools}
\usepackage{musicography}

\usepackage{thmtools}
\usepackage[nameinlink]{cleveref}

\DeclarePairedDelimiter\abs{\lvert\hspace{0.1ex}}{\rvert}

\theoremstyle{plain}

\newtheorem{thm}[equation]{Theorem}
\newtheorem{prop}[equation]{Proposition}
\newtheorem{lem}[equation]{Lemma} 
\newtheorem{cor}[equation]{Corollary}

\newtheorem*{cor*}{Corollary}
\newtheorem*{prob*}{Problem}
\newtheorem*{thm*}{Theorem}
\newtheorem*{thma*}{Theorem A}
\newtheorem*{thmb*}{Theorem B}

\theoremstyle{remark}
\newtheorem{exm}[equation]{Example}
\newtheorem{rmk}[equation]{Remark}

\theoremstyle{definition}
\newtheorem{defn}[equation]{Definition}

\newenvironment{enumalph}
{\begin{enumerate}}
{\end{enumerate}}

\DeclareMathOperator{\Gal}{Gal}

\DeclareMathOperator{\Nm}{Nm}

\DeclareMathOperator{\Repart}{Re}

\DeclareMathOperator{\Spec}{Spec}

\DeclareMathOperator{\Gm}{\mathbb{G}_m}

\newcommand{\dD}{\mathrm{d}}

\newcommand{\aunder}{\pmb{a}}
\newcommand{\bunder}{\pmb{b}}

\newcommand{\defi}[1]{\emph{\textsf{#1}}} 				

\newcommand{\A}{\mathbb A}
\newcommand{\C}{\mathbb C}
\newcommand{\F}{\mathbb F}

\newcommand{\PP}{\mathbb P}
\newcommand{\Q}{\mathbb Q}

\newcommand{\Z}{\mathbb Z}
\newcommand{\m}{\mu}
\newcommand{\eps}{\varepsilon}

\usepackage{tipx}

\newcommand{\qstroke}{%
  \text{\ooalign{\hidewidth\raisebox{-0.1ex}{-\kern-.25em}\hidewidth\cr$q$\cr}\hspace{0.2ex}}%
}
\newcommand{\qq}{q^\times}
\newcommand{\ZZ}{\mathbb Z}

\newcommand{\psmod}[1]{~(\textup{\text{mod}}~{#1})}

\newcommand{\frakp}{\mathfrak{p}}

\allowdisplaybreaks

\begin{document} 

\title{Hypergeometric motives from Euler integral representations}

\author{Tyler L. Kelly}
 \address{School of Mathematical Sciences, Queen Mary University of London, 327 Mile End Rd, London E1 4NS, UK}
\email{t.l.kelly@qmul.ac.uk}

\author{John Voight}
\address{Department of Mathematics, Dartmouth College, 6188 Kemeny Hall, Hanover, NH 03755, USA; School of Mathematics and Statistics, Carslaw Building (F07), University of Sydney, NSW 2006
Australia}
\email{jvoight@gmail.com}

\date{\today}
\setcounter{tocdepth}{1}

\begin{abstract}
We revisit certain one-parameter families of affine covers arising naturally from Euler's integral representation of hypergeometric functions. We introduce a partial compactification of this family. We show that the zeta function of the fibers in the family can be written as an explicit product of $L$-series attached to nondegenerate hypergeometric motives and zeta functions of tori, twisted by Hecke Grossencharacters.  This permits a combinatorial algorithm for computing the Hodge numbers of the family.
\end{abstract}

\maketitle

\tableofcontents

\section{Introduction}

\subsection{Motivation}

The world of hypergeometric functions invites a rich interplay between complex analysis, algebraic geometry, number theory, and physics guided by the hypergeometric differential equation.  The simplest version of this is found in the Legendre family of elliptic curves, defined by
\[ y^2 = x(x-1)(x-t). \]
The period associated to this family can be computed to be hypergeometric, following Euler~\cite{Eul48}. Two centuries later, Igusa found an arithmetic analogue to Euler's integral formula: the point counts on the Legendre family of elliptic curves can be related to a (truncated) hypergeometric equation~\cite{Igusa}.  Starting in the 1980s, Katz gave a vast generalization \cite{Katz96} in his theory of motives attached to rigid local systems.  Since then, the literature has proposed many families of varieties that exemplify hypergeometric motives  in varying generality: see Roberts--Rodriguez-Villegas \cite[\textsection 3]{RRV} for an overview and the recent monograph by Fuselier--Long--Ramakrishna--Swisher--Tu \cite{FLRST} for a history and further references.  

We focus on a symmetrized version of a generalization of the Legendre family above.  Let $n,m \geq 1$ and $a_1,\dots,a_n,b_1,\dots,b_n \in \Z$.  Inspired by Euler's formula, we start with the family of affine cyclic covers over $\Z[1/m]$
\begin{equation} \label{eqn:thefam-intro}
\begin{aligned}
Y = Y_{\aunder,\bunder,m} \colon y^m &= f_{\aunder,\bunder}(x)  \colonequals \prod_{i=1}^{n} (-x_i)^{a_i} (1-x_i)^{b_i-a_i} \\
1 &= tx_1 x_2 \cdots x_n \\
0,1 &\neq x_1,x_2,\dots,x_n
\end{aligned}
\end{equation}
in the parameter $t$ for $\PP^1 \smallsetminus \{0,1,\infty\}$. We note that our generalization has introduced a new sign. This choice of sign yields substantial streamlining; its absence already shows up as an obtrusive quadratic character when counting points on the Legendre family.  

The fibers $Y_t$ of this family are smooth of dimension $n-1$.  When $a_i-b_j \not \in m\Z$ for all $i,j=1,\dots,n$, Euler's integral formula directly implies at least one period of $Y$ is a hypergeometric function in the parameter $t$, namely $F(a_1/m,\dots,a_n/m; b_1/m,\dots,b_n/m;t)$ (\Cref{lem:eulerperiod}), a feature not as apparent in other hypergeometric families.  Other periods, however, may be less transparent.

We turn to the arithmetic story for $Y$ and computing $\#Y_t(\F_q)$. To do so, we first look at defining an analogue of hypergeometric functions over finite fields. The literature contains multiple definitions and normalizations~\cite{Greene, DM, Katz, BCM, FLRST}, each having their own strengths and contexts.
In particular, Fuselier--Long--Ramakrishna--Swisher--Tu \emph{define} a new finite field analogue of the hypergeometric function in a recursive way to count the number of $\F_q$-rational points in the affine variety defined by the first two equations in~\eqref{eqn:thefam-intro} directly \cite[Proposition 4.2]{FLRST}. The other definitions of hypergeometric functions differ, but under a nondegeneracy condition (namely, $\gcd(m,b_i-a_i)=1$ for all $i$), the hypergeometric function in \cite{FLRST} coincides with that of, say \cite{DM}.

However, the interpretation of this formula can run into problems when there are degeneracies. This happens more than one expects, as degenerate hypergeometric parameters arise when $a_i - b_i \in d\Z$ for some $d \mid m$. In these cases, the formula in \cite{FLRST} uses their recursive formulae with degenerate parameters. In this paper, we introduce a new model $X$ that is a partial compactification of $Y$ (see \textsection\ref{subsec: Partial Compactification}) whose motive avoids degenerate hypergeometric motives. Indeed, its $L$-series is an explicit product of $L$-series attached to nondegenerate hypergeometric motives and zeta functions of tori twisted by a Hecke Grossencharacter. 
In turn, the mixed Hodge numbers for $X$ can be computed directly from this motivic perspective~\cite{CortiGolyshev, Fedorov, RodVill}.

\subsection{Results}

Let $t \in \PP^1(\Q) \smallsetminus \{0,1,\infty\}$ and write $S$ for the set of primes dividing $m$ together with the primes dividing the numerator or denominator of $t$ or $t-1$.  Write $Y_t \colonequals Y_{\aunder,\bunder,t}$ for the fiber over $t$ in the family \eqref{eqn:thefam-intro} over $\Spec \Z[S^{-1}]$ (inverting the primes in $S$).  For $p \not \in S$, write $Y_{t,\F_p}$ for the base change to $\F_p$ and take 
 \begin{equation} 
 Z(Y_{t,\F_p},T) \colonequals \exp\left(\sum_{r=1}^{\infty} \#Y_t(\F_{p^r}) \frac{T^r}{r}\right) \in (1+T\Z[[T]]) \cap \Q(T).
 \end{equation}
Let
 \begin{equation} 
 \zeta_S(Y_t,s) \colonequals \prod_{p \not\in S} Z(Y_{t,\F_p},p^{-s})^{-1}
 \end{equation}
be the zeta function attached to $Y_t$; then $\zeta_S(Y_t,s)$ is convergent in a right half-plane for $s \in \C$ (e.g.\ Serre \cite[\S 1.5]{Serre:nxp}).  

Roughly speaking, our main result will show that there is a partial compactification $X_t \supset Y_t$ whose zeta function, similarly defined, is an explicit product of $L$-series attached to nondegenerate hypergeometric parameters and twists of zeta functions of tori by Hecke Grossencharacters.

We proceed by making this precise by introducing our notation.  Let $\pmb{\alpha}=(\alpha_1,\dots,\alpha_n) \in (\Q/\Z)^n$ and $\pmb{\beta}=(\beta_j)_j \in (\Q/\Z)^n$ be \defi{hypergeometric parameters}.  We say that $\alpha_i$ or $\beta_i$ is a \defi{degenerate parameter} if $\beta_i-\alpha_i \in \Z$. The hypergeometric parameters $\pmb \alpha, \pmb \beta$ are \defi{degenerate} if there exists a degenerate parameter, and \defi{nondegenerate} otherwise.  Finally, we say the parameters are \defi{isotypically degenerate} for $\gamma$ if the set of degenerate parameters in $\Q/\Z$ is $\{\gamma\}$, and we let $e(\pmb{\alpha},\pmb{\beta})$ be its multiplicity in $\pmb \alpha$.  (See \Cref{exm:isotypic}.)

Suppose that $\pmb{\alpha},\pmb{\beta}$ are nondegenerate.  Let $m \in \Z_{\geq 1}$ be minimal such that $m\pmb{\alpha},m\pmb{\beta} \subset \Z$.  In \cref{sec:hypgem}, we recall the definition of the \emph{period-normalized} hypergeometric $L$-series 
\begin{equation}
L_S(H(\pmb{\alpha},\pmb{\beta},t), \Q(\zeta_m), s),
\end{equation}
defined by an Euler product and convergent in a right-half plane.  We note that this $L$-series depends on the ordering of the parameters, but compensates by being invariant under shifts (\Cref{rmk:whyperiodtwist}).  Based on deep results of Katz, we explain that the period-normalized hypergeometric $L$-series for nondegenerate parameters has degree $n\phi(m)$ over $\Q$ (\Cref{Katzthm}).  

The variety $Y_{\aunder,\bunder,m}$ is a branched cover of the affine variety $U$ defined by $1=tx_1\cdots x_n$ and $0,1 \neq x_1,\dots,x_n$.  We partially compactify to $V \supseteq U$ defined by $1=tx_1\cdots x_n$ and $0 \neq x_1,\dots,x_n$ to provide a smooth branched cover of a toric hypersurface, as follows.  

\begin{thm} \label{thm:mainthm}
There exists a partial $\mu_m$-equivariant compactification $X_{\aunder,\bunder,m} \supseteq Y_{\aunder,\bunder,m}$ over $V \supseteq U$ such that 
for all $t \in \PP^1(\Q) \smallsetminus \{0,1,\infty\}$, 
\[ \zeta_S(X_{\aunder,\bunder,m,t},s) = \prod_{d \mid m} 
\begin{cases}
L_S(H(\aunder/d,\bunder/d,t),\Q(\zeta_d),s), & \text{if $\aunder/d,\bunder/d$ is nondegenerate;} \\
\zeta_S((\Gm)^{e-1}, \Q(\zeta_d), s, \psi_{\aunder,\bunder,d,t})^{(-1)^{n-e}}, & \begin{aligned}
    &\parbox[t]{0.35\linewidth}{if $\aunder/d,\bunder/d$ is isotypically degenerate, where $e=e(\aunder/d,\bunder/d)$;}
  \end{aligned} \\
1, & \text{else};
\end{cases} \]
where $\psi_{\aunder,\bunder,d,t}$ is an explicit Hecke Grossencharacter.
\end{thm}

\Cref{thm:mainthm} describes the full (mixed) motive of the family of varieties whose periods arise from Euler's integral formula.  The Hecke Grossencharacters arising in \Cref{thm:mainthm} are described in \eqref{eqn:explicithypergross}; we briefly recall in \cref{sec:gross} how Hecke Grossencharacters may be obtained from Jacobi sums (after Weil and Anderson).

If one tries to extend the hypersurface defined by $y^m=f_{\pmb a, \pmb b}(x)$ across $V$ without modification, then the variety is badly singular. We require a more elaborate partial compactification to obtain \Cref{thm:mainthm}, see \eqref{eqn:alleqns}.  In a nutshell, we organize the factors of $f(x)$ according to the greatest common divisor of their multiplicity with $m$, and show that these glue together appropriately (\Cref{prop:openimmers}).  

As the Hodge numbers of nondegenerate hypergeometric motives and Hecke Grossencharacters are known, as a corollary we conclude that the Hodge numbers of the family $X_t$ are effectively computable, in particular the degrees of the $L$-series appearing in $\zeta(X_t,s)$.  Indeed, for the nondegenerate case, one can find the Hodge numbers associated to the hypergeometric motive associated to a BCM hypergeometric function, following the zigzag procedure \cite{RRV}.  In the isotypically degenerate case, we obtain a torus factor twisted by a Hecke Grossencharacter, which can then be computed \cite{Watkins:comput}.

Our method of proof uses character sums.  In the nondegenerate case, we give a new non-inductive proof (\Cref{main hypergeometric formula}).  But the heart of our work is to show a certain \emph{M\"obius-like cancellation}: terms that show up above fibers where $x_i=1$ in the degenerate case either line up when isotypically degenerate and cancel otherwise (\Cref{hypergeometric main}).

\subsection{Precise comparison to previous results} \label{sec:1compare}
There is a rich literature of where hypergeometric functions over finite fields have been shown to relate to rational point counts (see, e.g., the first paragraph of \cite{OnoforZagier} for a long list).
We now focus on the history of the family given in ~\eqref{eqn:thefam-intro}, where many authors have studied it in varying degrees of generality. As stated above, the intuition that the point counts on the Legendre family of elliptic curves can be related to a (truncated) hypergeometric equation corresponding to the family's period can be traced back to work of Igusa~\cite{Igusa}. Its relation to finite field hypergeometric functions has been articulated in numerous ways (see \cite[\S4]{Koike}, \cite[Theorem 1]{Ono}, \cite[Theorem 2]{Rouse}), then proven for a new example \cite{Goo18} by considering a quadruple cover. 

More recently, Deines, Fusilier, Long, Swisher, and Tu analyzed the family 
\begin{equation}\label{eqn:DFLST}
y^n = (x_1\cdots x_{n-1})^{n-1}(1-x_1) \cdots (1-x_{n-1})(x_1 - \lambda x_2\cdots x_{n-1})
\end{equation}
as a higher-dimensional version of the Legendre curves, and show in \cite[Theorem 2]{DFLST} that its rational point count is the sum of a polynomial and finite-field hypergeometric functions as defined by Greene \cite{Greene}.  Fusilier, Long, Ramakrishna, Swisher, and Tu then generalized this result \cite{FLRST} to the case
\begin{equation} \label{eqn:FLRST}
y^m = x_1^{a_1}(1-x_1)^{b_1} \cdots x_{n-1}^{a_n}(1-x_{n-1})^{b_{n-1}}(1 - t x_1x_2\cdots x_{n-1})^{b_n}.
\end{equation}
(Compare also with Katz \cite[Theorem 8.4.1]{Katz96}.)

Over $\F_q$, when $q$ is a prime power so that $q \equiv 1 \pmod m$, they prove \cite[Proposition 4.2]{FLRST} that the $\F_q$-rational point counts correspond to a polynomial in $q$ and a sum of finite field hypergeometric functions. However, to do so, they created a new recursive definition for finite field hypergeometric functions (see (4.4) of loc. cit.) to allow them to use character theory to directly imply the result. 

Their finite field hypergeometric function definition uses a \emph{period normalization}, and therefore differs from those given in the literature by previous authors Greene \cite{Greene}, McCarthy \cite{DM}, Katz \cite{Katz}, and Beukers--Cohen--Mellit \cite{BCM}.  The authors \cite[\S 4.4]{FLRST} provide a relation between their hypergeometric function and that given in Greene's paper.  Also, they related theirs to that given by McCarthy, but only in the so-called \emph{primitive} case \cite[Definition 4.3]{FLRST} when $a_i - b_j \not \in m\Z$ for all $i,j=1,\dots,n$.  (Note that primitive implies nondegenerate, but not conversely.)  

Given these differences, we found it less than straightforward to combine results from these papers.  We navigate around this obstacle by restricting ourselves to only allowing formulas involving the nondegenerate hypergeometric functions, where there is more agreement and the theorem of Katz allows us to recognize the associated $L$-series as motivic. Indeed, one may interpret our main result as saying that the $L$-series of a degenerate hypergeometric motive arises from a Hecke Grossencharacter.  We found that the period normalization (following Fusilier--Long--Ramakrishna--Swisher--Tu) gives the simplest formulation in our main result (\Cref{thm:mainthm}); but the choice of signs is essential.

Stepping back a bit further, there are other important families which realise hypergeometric motives.  When the parameters are \emph{defined over $\Q$}---meaning that for all $k \in (\Z/m\Z)^\times$ we have $k\pmb{\alpha}=\alpha$ and $k\pmb{\beta}=\beta$ as multisets---there is a particularly nice setup: see the toric description of \emph{source varieties} described by Roberts--Rodriguez-Villegas \cite[\S 3]{RRV} and related explicitly to the models defined by Beukers--Cohen--Mellit \cite{BCM}.  However, these models are not available when the parameters are not defined over $\Q$, and so this hypothesis is substantial (the Klein--Mukai pencil of K3 surfaces in $\PP^3$ defined by $x^4+y^3z+z^3w+w^3y-4txyzw=0$ \cite{DKSSVW20} already sees hypergeometric motives not defined over $\Q$).  

Lastly, we remark the family \eqref{eqn:FLRST} is affine, and therefore a singular compactification of the starting family $Y_{\aunder,\bunder,m}$ (our equation has been symmetrized).  It is unclear how to proceed to understand the geometry of these families (for example, to understand the Hodge numbers) without a partial desingularization and compatification.  In sum, then, our contribution is to define a nicer geometric realization of hypergeometric motives that is available for all parameters and permits the computation of the zeta function over $\Q$ as a product of $L$-series of nondegenerate hypergeometric motives and Hecke Grossencharacters.  

\subsection{Acknowledgements}

The authors would like to thank Asem Abdelraouf, Giulia Gugiatti, Nalini Joshi, Daniel Kaplan, Albrecht Klemm, David Roberts, Fernando Rodriguez Villegas, and Wadim Zudilin for discussions relating to this work. Certain calculations used in this paper were done in Magma \cite{Magma}.  Kelly acknowledges support from EPSRC Grant EP/S03062X/1, the UK Research and Innovation Future Leaders Fellowship MR/T01783X/1, and the hospitality of Dartmouth College.  Voight was supported by a Simons Collaboration grant (550029) and would like to thank the hospitality of the Abdus Salam International Centre for Theoretical Physics (ICTP), where some of this research was undertaken as part of the \emph{Workshop on Number Theory and Physics} in June 2024.  

\section{Hypergeometric functions and integral representations}

For motivation, we begin with the classical, complex theory of (generalized) hypergeometric functions.  We present this standard material in a symmetric way.

\subsection{Differential equation}

Let $(x)_k$ be the rising factorial defined for $x \in \C$ and $k \in \Z_{\geq 0}$ by $(x)_0 \colonequals 1$ and
\[ (x)_k \colonequals x(x+1)\cdots(x+k-1)=\frac{\Gamma(x+k)}{\Gamma(x)} \]
for $k>0$, where $\Gamma$ is the usual complex $\Gamma$-function.  

\begin{defn} \label{defn:hypergeometricFunction}
Let $n\in\Z_{\geq 1}$, let $\pmb{\alpha}=(\alpha_1,\dots,\alpha_n) \in \Q^n$ and $\pmb{\beta}=(\beta_1,\dots,\beta_n)\in (\Q_{>0})^n$; we call $\pmb{\alpha}$ the \defi{numerator parameters} and $\pmb{\beta}$ the \defi{denominator parameters}.  The \defi{(generalized) hypergeometric function} is the formal series 
\begin{equation}
F(\pmb{\alpha},\pmb{\beta}, t) \colonequals \sum_{k=0}^{\infty}\frac{(\alpha_1)_k\cdots(\alpha_n)_k}{(\beta_1)_k\cdots(\beta_n)_k}t^k \in \Q[[t]]. \label{E:hyperg}
\end{equation}
\end{defn}

Consider the differential operator $\theta \colonequals t \displaystyle{\frac{\dD}{\dD t}}$.
We define the \defi{hypergeometric differential operator} for parameters $\pmb{\alpha},\pmb{\beta}$ to be
\begin{equation} \label{eqn:Dab'lldoit}
D(\pmb{\alpha},\pmb{\beta},t) \colonequals (\theta +\beta_1 -1)\cdots(\theta + \beta_n -1) - t (\theta+\alpha_1)\cdots(\theta + \alpha_n). 
\end{equation}
When $\beta_j=1$ for some $j$, the hypergeometric function $F(\pmb{\alpha},\pmb{\beta}, t)$ is annihilated by $D(\pmb{\alpha},\pmb{\beta},t)$.  In general, we have the following.

\begin{lem} \label{lem:shiftinga1d}
For every $j=1,\dots,n$, the operator $D(\pmb{\alpha},\pmb{\beta},t)$ annihilates the function
 \[
 F_j(\pmb{\alpha},\pmb{\beta},t) \colonequals t^{1-\beta_j}F(\pmb{\alpha} + (1-\beta_j), \pmb{\beta} + (1-\beta_j), t),
 \]
 where $\pmb{\alpha} + (1-\beta_j) = (\alpha_1+1-\beta_j,\dots,\alpha_n+1-\beta_j)$ and similarly for $\pmb{\beta}+(1-\beta_j)$. 
\end{lem}

\begin{proof} Direct computation.\end{proof}

\begin{rmk}
When $\#\pmb{\beta}=n$ (i.e., the denominator parameters are all distinct), \Cref{lem:shiftinga1d} gives a basis of solutions to the differential equation $D(\pmb{\alpha},\pmb{\beta},t)F=0$ \cite[\S 2.1.2]{slater} (see also \cite{Levelt1}).  Otherwise, one may need solutions involving the logarithm \cite{Levelt2}.
\end{rmk}

\subsection{Integral representation}

Euler \cite{Eul48} provided an integral representation for the hypergeometric function with parameters $\pmb{\alpha}=\{a,b\}$ and $\pmb{\beta}=\{c,1\}$, namely
\begin{equation}
F(a,b;c,1;t) = \frac{\Gamma(c)}{\Gamma(b)\Gamma(c-b)}\int_0^1 x^{b-1}(1-x)^{c-b-1}(1-tx)^{-a}\,\mathrm{d}x
\end{equation}
whenever $c>b>0$ and for all $t \in \C$ with $\abs{t}<1$.  This formula can be proven by expanding $(1-tx)^{-a}$ using the binomial theorem and integrating term-by-term to reduce to Euler's integral formula for the $\beta$-function.  This integral formula then inductively generalizes as follows.  In light of \Cref{lem:shiftinga1d}, \emph{ordering} the parameters we may suppose without loss of generality that $\beta_n=1$.  

\begin{lem} \label{lem:euler}
Suppose that $\alpha_i-\beta_j \not\in \Z$ for all $i,j$ and that $\beta_n=1$.  Then we have the equality
\begin{align*}
F(\pmb{\alpha},\pmb{\beta},t) 
&= \prod_{i=1}^{n-1} \frac{\Gamma(\beta_i)}{\Gamma(\alpha_i)\Gamma(\beta_i-\alpha_i)} \int_0^1 \cdots \int_0^1 
(1-tx_1\cdots x_{n-1})^{-\alpha_n} \prod_{i=1}^{n-1} x_i^{\alpha_i-1} (1-x_i)^{\beta_i-\alpha_i-1} \,\mathrm{d}x_i \\
&= e^{\pi i \alpha_n} \gamma_{\pmb{\alpha},\pmb{\beta}}\underset{tx_1\cdots x_n = 1}{\int_0^1 \cdots\int_0^1} x_n \prod_{i=1}^{n} x_i^{\alpha_i-1} (1-x_i)^{\beta_i-\alpha_i-1} \,\mathrm{d}x_1 \cdots \mathrm{d}x_{n-1}
\end{align*}
in a domain of convergence (for example, $\abs{\arg(1-t)} < \pi$ and $\Repart(\beta_i)>\Repart(a_i)>0$ for all $i=1,\dots,n-1$),
where
\begin{equation} 
\gamma_{\pmb{\alpha},\pmb{\beta}} \colonequals \prod_{i=1}^{n} \frac{\Gamma(\beta_i)}{\Gamma(\alpha_i)\Gamma(\beta_i-\alpha_i)} 
\end{equation}
(and in the second integral, we mean the restriction of the differential).
\end{lem}

\begin{proof}
The domain of convergence and first line of the equality can be found in Nesterenko \cite[Lemma 2]{Nest} applying Slater \cite[(4.1.3)]{slater}; the second equality comes from the relation $\Gamma(z)\Gamma(1-z)=\pi/\sin(\pi z)$ and substitution.  
\end{proof}

\begin{rmk}
Although the hypergeometric function is visibly independent of permutation of the parameters, the integral representation is not: at least in its symmetrized form, it depends on a matching between numerator and denominator pairings.

In the below, we work with a version with a different choice of signs; this changes the constant $\gamma_{\pmb{\alpha},\pmb{\beta}}$ by a root of unity, something which is already a bit complicated to analyze precisely in the branched cover.
\end{rmk}

\section{Birational models for the integral representation}

In this section, using the integral representation \Cref{lem:euler} we define a family of smooth affine varieties with hypergeometric periods; we then give a partial compactification.  

\subsection{Starter model}

We start by elaborating upon the family \eqref{eqn:thefam-intro}, starting with a more complete definition.  
Our family is defined over the base $T \colonequals \PP^1 \smallsetminus \{0,1,\infty\}$ (initially over $\Z$, then defined over a base ring depending on the parameters).  We begin with the affine variety $U_n \subseteq T \times_{\Z} \A_{\Z}^n$ defined by
\begin{equation} \label{eqn:affop}
\begin{aligned}
tx_1 x_2 \cdots x_n &= 1 \\
x_1,x_2,\dots,x_n &\neq 0,1. 
\end{aligned}
\end{equation}
When $n$ is understood, we suppress notation and write $U$ for $U_n$. The fibers of the map $U \to T$ are contained in the torus $(\Gm)^n \subset \A^n$, in fact in the open where the axes $x_i=1$ are removed.

We now define a family of branched covers of $U$.  Let $\aunder=(a_1,\dots,a_n)$ and $\bunder=(b_1,\dots,b_n)$ with $a_i,b_j \in \Z$.   Define
\begin{equation} \label{eqn:thefam-f}
f_{\aunder,\bunder}(x_1,\dots,x_n)  \colonequals \prod_{i=1}^{n} (-x_i)^{a_i} (1-x_i)^{b_i-a_i} \in \Z[x_1,\dots,x_n,(x_1(1-x_1)\cdots x_n(1-x_n))^{-1}].
\end{equation}
Let $m \in \Z_{\geq 1}$.  We now define the hypersurface in $U \times \A^1$
\begin{equation} \label{eqn:thefam}
Y = Y_{\aunder,\bunder,m} \colon y^m = f_{\aunder,\bunder}(x_1,\dots,x_n)
\end{equation}
together with the natural branched cover $Y \to U$.   

\begin{lem} \label{lem:cmon}
For any commutative ring $R$ with $1 \neq 0$ and $t \in R$ such that $t,t-1 \in R^\times$, the fiber $Y_t$ of the family \eqref{eqn:thefam} is smooth of dimension $n-1$.  
\end{lem}

\begin{proof}
Immediate, as we have removed the branch locus.
\end{proof}

It is an important step in the theory of motives to recognize the corresponding period integrals, as follows.

\begin{lem} \label{lem:eulerperiod}
Suppose that $a_i-b_j \not \in m\Z$ for all $i,j=1,\dots,n$.  Let $c_1,\dots,c_n,d_1,\dots,d_n \in \Z$, and suppose $c_n=-d_n$.  Then for all $k \in (\Z/m\Z)^\times$, and for all $t \in \C$, there exists a cycle $Z_t$ on $Y$ such that 
\begin{equation} 
\int_{Z_t} \frac{\prod_{i=1}^n (-x_i)^{c_i} (1-x_i)^{d_i}}{y^k}\,\mathrm{d}x_1\,\cdots\,\mathrm{d}x_{n-1}  = \nu \gamma_{\pmb{\alpha},\pmb{\beta}} t^{k b_n/m} 
F(\pmb{\alpha}, \pmb{\beta}, t)
\end{equation}
under an appropriate branch cut, where 
\begin{equation}
\nu = \frac{\pi e^{\pi i \sum_{i=1}^n (c_i-ka_i/m)}}{\sin(k\pi(b_n-a_n)/m))}
\end{equation}
and
\begin{equation} \label{eqn:alphabetam}
\begin{aligned}
\pmb{\alpha} &\colonequals \left(c_1+1+\frac{k(b_n-a_1)}{m},\dots,c_{n-1}+1+\frac{k(b_n-a_{n-1})}{m},c_n+\frac{k(b_n-a_n)}{m}\right);  \\
\pmb{\beta} &\colonequals \left(2+c_1+d_1+\frac{k(b_n-b_1)}{m},\dots,2+c_{n-1}+d_{n-1}+\frac{k(b_n-b_{n-1})}{m},1\right).
\end{aligned}
\end{equation}
\end{lem}
\begin{rmk}
    We refer the reader to \cite[\textsection2.2, 2.7]{GKZ} for a description of the cycle $Z_t$.
\end{rmk}
\begin{cor}
With notation as in \Cref{lem:eulerperiod}, we have
\begin{equation} 
\int_{Z_t} \frac{\prod_{i=1}^n (-x_i)^{a_i-b_n-1} (1-x_i)^{b_i-a_i-1}}{y^{m-1}}\,\mathrm{d}x_1\,\cdots\,\mathrm{d}x_{n-1}  = \nu \gamma_{\pmb{\alpha},\pmb{\beta}}
t^{(m-1)b_n/m} F(\pmb{\alpha}, \pmb{\beta}, t)
\end{equation}
where
\begin{equation} \label{eqn:alphabetam-onthenose}
\begin{aligned}
\pmb{\alpha} &\colonequals \left(\frac{a_1-b_n}{m},\dots,\frac{a_{n-1}-b_n}{m},\frac{a_n-b_n}{m}-1\right);  \\
\pmb{\beta} &\colonequals \left(\frac{b_1-b_n}{m},\dots,\frac{b_{n-1}-b_n}{m},1\right).
\end{aligned}
\end{equation}
\end{cor}

\begin{proof}
Substitute $k=m-1$, $c_i=a_i-b_n-1$ and $d_i=b_i-a_i-1$ for $i=1,\dots,n$.
\end{proof}


It is an expression of the interplay mentioned in the introduction that the periods as expressed by the hypergeometric series above match the point as expressed by the finite field hypergeometric functions in our main result.  

\begin{proof}[Proof of \Cref{lem:eulerperiod}]
We compute that (after an appropriate branch cut),
\begin{equation} \label{eqn:ymm1}
\begin{aligned}
   \frac{\prod_i (-x_i)^{c_i} (1-x_i)^{d_i}}{y^k} &= \prod_{i=1}^n (-x_i)^{c_i-ka_i/m}(1-x_i)^{d_i+k(b_i-a_i)(-1)/m} \\
   &= (-1)^{\sum_{i=1}^n (c_i-ka_i/m)} \prod_{i=1}^{n} x_i^{c_i-ka_i/m}(1-x_i)^{d_i-k(b_i-a_i)/m}.
\end{aligned}
\end{equation}
We now apply \Cref{lem:euler} and match exponents.  We are required to take $\beta_n=1$ for the analytic expansion, so $\beta_n-\alpha_n-1=-\alpha_n=d_n-k(b_n-a_n)/m$ so $\alpha_n = k(b_n-a_n)/m-d_n$.  But then 
\begin{equation}
\begin{aligned}
x_n^{c_n-ka_n/m} &= x_n^{c_n+k(b_n-a_n)/m - kb_n/m} = x_n^{c_n+d_n+\alpha_n} x_n^{-kb_n/m} \\
&= x_n^{\alpha_n} (tx_1\cdots x_{n-1})^{kb_n/m}. 
\end{aligned}
\end{equation}
Expansion and substitution in \eqref{eqn:ymm1} then gives
\begin{equation} \label{eqn:ymm2}
   (-1)^{\sum_{i=1}^n (c_i-ka_i/m)} t^{kb_n/m} \left(\prod_{i=1}^{n-1} x_i^{c_i+k(b_n-a_i)/m}(1-x_i)^{d_i-k(b_i-a_i)/m}\right) x_n^{\alpha_n} (1-x_n)^{-\alpha_n}.
\end{equation}
Now for $i < n$, the exponent match gives $\alpha_i-1=c_i+k(b_n-a_i)/m$ and $\beta_i-\alpha_i-1=d_i-k(b_i-a_i)/m$ which yields 
\begin{equation}
\beta_i=2+c_i+d_i+\frac{k(b_n-a_i)}{m} - \frac{k(b_i-a_i)}{m} = 2+c_i+d_i+\frac{k(b_n-b_i)}{m}
\end{equation}
giving the parameters in the statement.  We similarly match the coefficient in front.
\end{proof}


\subsection{Partial compactification}\label{subsec: Partial Compactification}

We now define a partial compactification (birational model) for the family $Y$ defined in \eqref{eqn:thefam}.  This model will admit an open immersion from $Y$, giving a partial compactification at points $(x_1,\dots,x_n,y)$ where some $x_i=1$.  

First off: without loss of generality (considering modulo $m$), we may and do suppose that $a_1,\dots,a_n,b_1,\dots,b_n \geq 0$ and $b_i \geq a_i$ for all $i$.  (We avoided making this hypothesis before now for uniformity in the description, but now we are trying to find a compactification so we need to clear denominators.  In principle, we only need to ask this for $b_i-a_i \geq 0$.)

We first subdivide the product defining $f_{\aunder,\bunder}$ according to a greatest common divisor with $m$.  For $d \mid m$ a (positive) divisor, define
\begin{equation}
f_{\aunder, \bunder, d}(x_1,\dots,x_n) \colonequals \prod_{\substack{i \\ \gcd(m,a_i)=d}} (-x_i)^{a_i / d} \prod_{\substack{i \\ \gcd(m,b_i-a_i)=d}} (1-x_i)^{(b_i-a_i)/d}. 
\end{equation}
Then
\begin{equation}
f_{\aunder,\bunder}(x_1,\dots,x_n) = \prod_{d \mid m} f_{\aunder,\bunder,d}(x_1,\dots,x_n)^d
\end{equation}
since all factors occur exactly once in the product, with the right exponent.

We work over $V \supseteq U$ defined by
\begin{equation}
\begin{aligned}
tx_1 x_2 \cdots x_n &= 1 \\
x_1,x_2,\dots,x_n &\neq 0;
\end{aligned}
\end{equation}
i.e., we close up the axes where a variable $x_i=1$.  However, our branched cover now comes with variables also indexed by divisors, i.e., we work in $(\Gm)^n \times \A^{\tau(m)}$ with toric variables $x_1,\dots,x_n$ and affine variables $y_d$ for $d \mid m$, where $\tau(m)$ is the number of divisors of $m$.  The branched cover $X=X_{\aunder,\bunder,m}$ is defined to be the subvariety in $V$ cut out by the system of equations
\begin{equation} \label{eqn:alleqns}
y_d^{d/h} = y_h \prod_{\substack{e \mid m \\ h \mid e \\ d \nmid e}} f_{\pmb{a}, \pmb{b}, e}(x_1,\dots,x_n)^{e/h},
\end{equation}
with an equation for all $h,d$ such that $h \mid d \mid m$, together with the equation $y_1=1$.  

\begin{prop} \label{prop:openimmers}
There is an open immersion
\begin{equation}
Y \hookrightarrow X
\end{equation}
defined by the identity on $x_1,\dots,x_n$ and
\begin{equation} \label{eqn:equivariant}
y_d = \frac{y^{m/d}}{\prod_{e : d \mid e \mid m} f_{\aunder,\bunder,e}(x_1,\dots,x_n)^{e/d}}
\end{equation}
for each $d \mid m$.  The birational inverse is defined by the identity on $x_1,\dots,x_n$ and
\begin{equation}
y = y_m f_{\aunder,\bunder,m}(x_1,\dots,x_n)^m.
\end{equation}
Moreover, the action of $\mu_m$ by $y \mapsto \zeta_m y$ extends to $X$ equivariantly using \eqref{eqn:equivariant}.
\end{prop}

\begin{proof}
One verifies that the map $X \dashrightarrow Y$ is an isomorphism (that is to say, defined) away from the points $(x_1,\dots,x_d,y)$ where $x_i=1$ for some $i=1,\dots,d$.  This map can be seen through considering \eqref{eqn:equivariant} in the case where $d = m$ and clearing denominators. 
\end{proof} 

For $I \subseteq [n] \colonequals \{1,\dots,n\}$, let
\begin{equation} \label{aIbImI}
\begin{aligned}
\aunder_I &\colonequals (a_i)_{i \in I} \\
\bunder_I &\colonequals (b_i)_{i \in I} \\
m_I &\colonequals \gcd(\{b_i-a_i : i \not \in I\} \cup \{m\}).
\end{aligned}
\end{equation}
We define a `twist' of $Y_{\aunder_I,\bunder_I,m_I}$ by $(-1)^{\sum_{i \not\in I} a_i}$ as follows.  Define 
the (quasi-)affine variety in $\A^{\#I}$ with variables $\{x_i\}_{i \in I}$ defined by the equations
\begin{equation}\label{def: Y a b m I subset}
Y_{\aunder, \bunder,m, I}' \colon  \quad
\begin{aligned}
    y^{m_I} &= (-1)^{\sum_{i \notin I} a_i} \prod_{i \in I} (-x_i)^{a_i}(1-x_i)^{b_i - a_i} \\
    1 &= t \prod_{i \in I} x_i\\
    x_i &\neq 0, 1 \text{ for all $i \in I$}.
\end{aligned}
\end{equation}

\begin{lem} \label{lem:xi1}
Let $I \subseteq \{x_1,\dots,x_n\}$ be a subset.  Then the intersection of $X_{\aunder,\bunder,m}$ with the affine subvariety of $V$ defined to be the locus where $x_i \neq 1$ for $i \in I$ and $x_i=1$ for $i \not \in I$ is isomorphic to $Y_{\aunder, \bunder,m, I}'$.
\end{lem}

\begin{proof}
We look back at the defining equations \eqref{eqn:alleqns}.  First, we see that if $h \mid d$ and $y_h=0$, then $y_d=0$.  Let $d \mid m$.  Take $h=1$ to get
\begin{equation} 
\begin{aligned}
y_d^d &= y_1 \prod_{\substack{e \mid m \\ d \nmid e}} f_{\aunder,\bunder,e}(x_1,\dots,x_n)^e \\
&= \prod_{\substack{e \mid m \\ d \nmid e}} \prod_{\substack{i \\ \gcd(m,a_i)=e}} (-x_i)^{a_i} \prod_{\substack{i \\ \gcd(m,b_i-a_i)=e}} (1-x_i)^{b_i-a_i}. 
\end{aligned}
\end{equation}
So $y_d=0$ if and only if $(1-x_i) \mid f_{\aunder,\bunder,e}(x_1,\dots,x_n)$ for some $i \not\in I$ and $e \mid m$ with $d \nmid e$.  By definition, this happens if and only if there exists $i \not \in I$ such that $d \nmid e=\gcd(b_i-a_i,m)$.  By the contrapositive, this says $y_d \neq 0$ if and only if $d \mid \gcd(b_i-a_i,m)$ for all $i \not\in I$, i.e., $d \mid m_I$.   

Let $x_i'=x_i,1$ according as $i \in I$ or not.  Plugging in, this leaves
\[ y_{m_I}^{m_I} = \prod_{\substack{e \mid m \\ m_I \nmid e}} \prod_{\substack{i \\ \gcd(m,a_i)=e}} (-x_i')^{a_i} \prod_{\substack{i \\ \gcd(m,b_i-a_i)=e}} (1-x_i')^{b_i-a_i}. \]
(As confirmation, we see that if $i \not \in I$ then $m_I \mid \gcd(b_i-a_i,m)$, so the only terms $1-x_i'$ that appear have $i \in I$.)  This simplifies to 
\[
y_{m_I}^{m_I} = (-1)^{\sum_{i \not \in I} a_i} \prod_{i \in I} (-x_i)^{a_i} \prod_{i \in I} (1-x_i)^{b_i-a_i}. \]
(If $m_i \mid e$ then the twist by $(-1)^{e}$ can be absorbed into the isomorphism.)  Replacing $m$ by $m_I$, the birational maps in \Cref{prop:openimmers} then define an isomorphism to
$Y_{\aunder_I,\bunder_I,m_I}$.
\end{proof}

\section{Finite field hypergeometric functions} \label{sec:ffhg}

In this section, we recall and compare definitions of finite field analogues of complex hypergeometric functions \cite{Greene, Katz, DM, FLRST, BCM}.

\subsection{Setup}

Let $q=p^r$ be a prime power and let
\begin{equation}
\qq \colonequals q-1.
\end{equation}  
We write $\omega^{1/\qq} \colon \F_q^\times \rightarrow \C^\times$ for a generator of the character group on $\F_q^\times$, so we can write $\omega^\mu$ for $\mu \in \Q$ such that $\qq\mu \in \Z$.  We extend $\omega \colon \F_q \rightarrow \C^\times$ by setting $\omega(0) = 0$ including for the trivial character $\eps$.  Further, let $\Theta \colon \F_q \rightarrow \C^\times$ be a nontrivial (additive) character. 

\begin{defn}\label{Gauss sum def}
For $\alpha \in (1/\qq)\ZZ$, we define the \defi{Gauss sum}
\begin{equation}
    g(\alpha) \colonequals  \sum_{x \in \F_q^\times} \omega^{\alpha}(x) \Theta(x).
\end{equation}
\end{defn}
The Gauss sum takes values in $\Q(\zeta_{q-1},\zeta_q) \subseteq \C$ and is well-defined on $k \in \bigl(\frac{1}{\qq}\ZZ\bigr)/\ZZ$.


\begin{rmk}\label{rmk:Gaussum}
While the definition of the Gauss sum also depends on the choice of additive character $\Theta$, the definitions for the Jacobi sum in~\Cref{def:jacobi}, the definition of the Hecke Grossencharacter in~\eqref{ourheckegross}, and the period normalized hypergeometric function in~\Cref{defn:ourtwisthgm} do not.
\end{rmk}

\begin{rmk}
Note the Gauss sum in this definition sums over the elements in $\F_q^\times$ rather than $\F_q$. The literature in finite field hypergeometric sums is not uniform; there can be a discrepancy when one defines $\eps(0)$ to be nonzero or defines the set of summation in the Gauss sum to include $0$ (over $\F_q$, instead of $\F_q^\times$) in the Gauss sum.  First, some authors impose that $\eps(0) = 0$  \cite{FLRST, DM} while others use the convention that $\eps(0) = 1$ \cite{IR90, BCM, Cohen2}.  Second, some authors will alter the summation in the definition of a Gauss sum to be over the elements in $\F_q$ rather than $\F_q^\times$ (e.g., \cite{IR90,DM} sum over $\F_q$ but \cite{FLRST, BCM} sum over $\F_q^\times$).    
\end{rmk}

With this definition for Gauss sum, we have the following standard properties:
\begin{prop}\label{std gauss formula}
We have that $g(0)=-1$ and, for $\alpha \in (1/\qq)\Z \smallsetminus \Z$,
\begin{equation}\label{gen gauss formula}
g(\alpha) g(-\alpha) = \omega^\alpha(-1) q.
\end{equation}
\end{prop}

\begin{proof}
See Cohen \cite[Lemma 2.5.8 and Proposition 2.5.9]{Cohen2}. 
\end{proof}

\subsection{Jacobi sums} \label{sec:jacobi}

We recall a definition for Jacobi sums.

\begin{defn} \label{def:jacobi}
Let $\alpha,\beta \in (1/\qq)\Z$. We define the \defi{Jacobi sum} 
$$
J(\alpha,\beta) = \sum_{x \in \F_q\setminus \{0,1\}} \omega^\alpha(x)\omega^\beta(1-x).
$$
\end{defn}

\begin{rmk}
Again, there are different conventions for the set of summation of a Jacobi sum, which can lead to discrepancies just as in the case of Gauss sums.  We have adopted the convention that $\eps(0)=0$, so the set of summation in \Cref{def:jacobi} can be extended to $x \in \F_q$, adding zero.
\end{rmk}

The rest of this subsection presents identities for Jacobi sums and Gauss sums that we find useful to compute $\#X(\F_q)$ in the next section.

\begin{prop}\label{BCM Jacobi Prop}
Let $\alpha, \beta \in (1/\qq)\Z$.  Then
\begin{equation}
J(\alpha, -\beta) = \frac{g(\alpha)g(-\beta)}{ g(\alpha-\beta)} + \begin{cases}   \omega^{-\beta}(-1) \qq, & \text{ if $\alpha - \beta \in \Z$}; \\ 0, &\text{ otherwise.} \end{cases}
\end{equation}
\end{prop}

\begin{proof}
We give a proof that is a straightforward modification of Ireland--Rosen \cite[\textsection 8.3, Theorem 1]{IR90} once we carefully track the consequences of taking $\eps(0)=0$ instead.  

We have the following expansion of the product $g(\alpha)g(-\beta)$:
\begin{equation}\begin{aligned}
g(\alpha)g(-\beta) &= \left(\sum_{x \in \F_q^\times} \omega^{\alpha}(x) \Theta(x)\right)\left( \sum_{y \in \F_q^\times} \omega^{-\beta}(y) \Theta(y)\right) \\
	&= \sum_{x,y\in\F_q^\times} \omega^{\alpha}(x) \omega^{-\beta}(y) \Theta(x+y) 
	= \sum_{u\in \F_q} \left(\sum_{\substack{x,y\in\F_q^\times \\ x+y=u}} \omega^{\alpha}(x) \omega^{-\beta}(y) \right)\Theta(u)
\end{aligned}\end{equation}
Note that when $u=0$ we can use the fact that $y = -x$ to obtain the following
\begin{equation}\begin{aligned}
\sum_{\substack{x,y\in\F_q^\times \\ x+y=0}} \omega^{\alpha}(x) \omega^{-\beta}(y) = \omega^{-\beta}(-1)\sum_{x\in\F_q^\times} \omega^{\alpha-\beta}(x) 
	= \begin{cases}  \omega^{-\beta}(-1) \qq & \text{ if $\alpha - \beta \in \Z$} \\ 0 &\text{ otherwise.}\end{cases}
\end{aligned}\end{equation}
Next, if $u \neq 0$ then we can rescale by a factor of $u$ to get the following
\begin{equation}\begin{aligned}
\sum_{\substack{x,y\in\F_q^\times \\ x+y=u}} \omega^{\alpha}(x) \omega^{-\beta}(y) = \sum_{\substack{x', y' \in \F_q^\times \\ x'+y' =1}} \omega^{\alpha}(ux') \omega^{-\beta}(uy') 
	= \omega^{\alpha - \beta}(u) \sum_{\substack{x', y' \in \F_q^\times \\ x'+y' =1}} \omega^{\alpha}(x') \omega^{-\beta}(y').
\end{aligned}\end{equation}
So then, if $\alpha - \beta \notin \Z$, we can simplify to get
\begin{equation}\begin{aligned}
g(\alpha)g(-\beta) =\sum_{u\in \F_q^\times} \omega^{\alpha - \beta} (u)\left(\sum_{\substack{x,y\in\F_q^\times \\ x+y=1}} \omega^{\alpha}(x) \omega^{-\beta}(y) \right)\Theta(u)
	= g(\alpha - \beta) \sum_{\substack{x,y\in\F_q^\times \\ x+y=1}} \omega^{\alpha}(x) \omega^{-\beta}(y).
\end{aligned}\end{equation}
Changing the summation to be over one variable, dividing by $g(\alpha - \beta)$, and using the fact that $g(0) = -1$, we obtain
\begin{equation}
\frac{g(\alpha)g(-\beta)}{ g(\alpha-\beta)} =  \sum_{x \in \F_q \setminus\{0,1\} }\omega(x)^{\alpha}\omega(1-x)^{\beta} - \begin{cases}   \omega^{-\beta}(-1) \qq, & \text{ if $\alpha - \beta \in \Z$}; \\ 0, &\text{ otherwise.} \end{cases}
\end{equation}
\end{proof}

\begin{cor}\label{cor:case by case Jacobi}
We have
\begin{equation}\label{case by case Jacobi}
\frac{g(\alpha)g(-\beta)}{ g(\alpha-\beta)} = \begin{cases}  J(\alpha, -\beta) & \text{ if $\alpha - \beta \notin \Z$}; \\ 
	-1 & \text{ if $\alpha \in \Z$ or $\beta \in \Z$}; \\
	-\omega^{\alpha}(-1)q & \text{ if $\alpha - \beta \in \Z$ but $\alpha\notin\Z$}.
\end{cases}
\end{equation}
\end{cor}

\begin{proof}
The first case is \Cref{BCM Jacobi Prop} and the latter two cases follow immediately from \Cref{std gauss formula}. 
\end{proof}

\begin{cor}\label{rest for degen case}
Let $\alpha, \beta \in \Q/\Z$ and suppose that $\alpha- \beta \in \Z$, but $\alpha \notin\Z$. Then
$$
J(\alpha, -\beta) = -\omega^{\alpha}(-1) = q^{-1} \frac{g(\alpha)g(-\beta)}{g(\alpha-\beta)}.
$$
\end{cor}
\begin{proof}
Since $\alpha-\beta \in \Z$, we have $\omega^{\alpha}(-1) = \omega^{-\beta}(-1)$. By combining \Cref{BCM Jacobi Prop} and \Cref{cor:case by case Jacobi}, we compute 
\[ J(\alpha, -\beta)  = -\omega^{\alpha}(-1)q +    \omega^{-\beta}(-1)\qq = \omega^{\alpha}(-1)(-q + \qq) = -\omega^{\alpha}(-1). \]
Applying \Cref{cor:case by case Jacobi} again proves the second identity.
\end{proof}

 The following lemma yields a relation for Jacobi sums induced that we use in the next section. This identity is known (e.g., see \cite[Lemma 2.8]{FLRST}), but we present a short geometric proof using a M\"obius transformation.


\begin{lem}\label{mobius trick}
    Let $\alpha, \beta\in (1/\qq)\Z$. We have 
    $$
    \sum_{x \in \F_q\setminus\{0,1\}} \omega^\alpha(x) \omega^\beta(1-x) =\sum_{x\in\F_q \setminus\{0,1\}} \omega^{\alpha}(-x) \omega^{-\beta-\alpha}(1-x).
    $$
\end{lem}
\begin{proof}
We apply the M\"obius transformation $M(x) = x/(x-1)$ to the left hand side to rearrange the sum. Note the summing indexation does not change as $M(\{0,1,\infty\}) = \{0,1,\infty\}$. After computing that $1-M(x) = (1-x)^{-1}$ under this transformation, we compute 
\begin{align*}
\sum_{x\in \F_q\setminus\{0,1\}} \omega^\alpha(x) \omega^\beta(1-x) &= \sum_{x\in\F_q\setminus\{0,1\}} \omega^\alpha(x/(x-1)) \omega^{-\beta}(1-x) \\
    &=  \sum_{x\in\F_q \setminus\{0,1\}} \omega^\alpha(-x) \omega^{-\beta-\alpha}(1-x).\qedhere
\end{align*}
\end{proof}

\subsection{Hecke Grossencharacters} \label{sec:gross}
 
In this section, we briefly review how Gauss sums and Jacobi sums yield Hecke Grossencharacters (or Gr\"ossencharaktere) over cyclotomic fields.  This idea goes back to Weil \cite{Weil} and has been investigated by many authors; see Watkins \cite{Watkins} for an explicit version.  

Although the theory is quite general, we focus on the case of interest.  Let $m \in \Z_{\geq 1}$.  We work with $K \colonequals \Q(\zeta_m)$ with ring of integers $\Z_K=\Z[\zeta_m]$.  We consider $\Q(\zeta_m) \subset \C$ by taking $\zeta_m=\exp(2\pi i/m)$.  Let $\frakp \subseteq \Z_K$ be a prime ideal that is coprime to $m$, and write 
$q = \Nm(\frakp)$ for its norm and let $\F_\frakp \colonequals \Z_K/\frakp$ for its residue field.  Denote by
\begin{equation}
\begin{aligned} 
\chi_\frakp^{1/m} \colon \F_\frakp^\times \to \langle \zeta_m \rangle 
\end{aligned}
\end{equation}
the character uniquely characterized by the condition that
\[ \chi_\frakp^{1/m}(x) \equiv x^{(q-1)/m} \pmod{\frakp}. \]
We extend by zero to $\F_\frakp$.  The character $\chi_\frakp^{1/m}$ has order $m$, so under a choice of isomorphism $\Z_K/\frakp \simeq \F_q$ it corresponds to $\omega^{1/m}$ (see \Cref{rmk:Gaussum}).  

We may therefore extend our notation to this setting as follows.  Given $\alpha \in (1/m)\Z$, we define the \defi{Gauss sum} for $\frakp$ as
\begin{equation} 
g_\frakp(\alpha) \colonequals \sum_{x \in \Z_K/\frakp} \chi_\frakp^\alpha(x)\Theta(x).
\end{equation}
And for $\pmb{\alpha},\pmb{\beta} \subset (1/m)\Z$ satisfying $\#\pmb{\alpha}=\#\pmb{\beta}=n$, we define
\begin{equation} \label{ourheckegross}
\psi_{\pmb{\alpha},\pmb{\beta}}(\frakp) \colonequals \prod_{i=1}^n \frac{g_\frakp(\alpha_i)g_\frakp(-\beta_i)}{g_\frakp(\alpha_i-\beta_i)}. 
\end{equation}

\begin{thm} \label{thm:heckechar}
The assignment $\frakp \mapsto \psi_{\pmb{\alpha},\pmb{\beta}}(\frakp)$ defines a Hecke Grossencharacter $\psi_{\pmb{\alpha},\pmb{\beta}}$ for $\Q(\zeta_m)$ with the property that if $\sigma_k(\zeta_m)=\zeta_m^k$ for $k \in (\Z/m\Z)^\times$, then $\sigma_k (\psi_{\alpha,\beta}) = \psi_{k\alpha,k\beta}$.
\end{thm}

\begin{proof}
See Watkins \cite[\S 2.2]{Watkins}.  Galois equivariance can also be proven directly from the definitions using the fact that $\sigma_k\circ \omega = \omega^k$ and $\sigma_k\circ \Theta = \Theta$ (analogously to \cite[Lemma 4.9.1(a)]{DKSSVW20}).  
\end{proof}

\subsection{A formula for (period-normalized) hypergeometric functions} \label{sec:hypgem}

With the analogy between Gauss sums and the $\Gamma$-function in mind, but thinking in terms of periods, we make the following definition.

\begin{defn} \label{defn:ourtwisthgm}
Let $\pmb{\alpha},\pmb{\beta} \subset (1/\qq)\Z$ be hypergeometric parameters with $\#\pmb{\alpha}=\#\pmb{\beta}=n$.  For $t \in \F_q^\times$, we define the \defi{(period-normalized) hypergeometric sum} by
\begin{equation}\label{def:Hpsi}
H(\pmb{\alpha}, \pmb{\beta}, t) \colonequals  \frac{-1}{\qq} \sum_{\m=0}^{\qq -1} G(\pmb{\alpha}+ \m/\qq, \pmb{\beta}+ \m/\qq)\omega((-1)^n t)^{\m/\qq},
\end{equation}
where
\begin{equation}\label{alt Gtwist}
G(\pmb{\alpha}, \pmb{\beta}) \colonequals \prod_{i=1}^n \frac{g( \alpha_i)g(-\beta_i)}{g(\alpha_i - \beta_i)}.
\end{equation}
\end{defn}

In \Cref{defn:ourtwisthgm}, note that we do not require the standard assumption in complex hypergeometric functions that $\beta_n=1$.  However, one may reduce to this case by introducing a twist, shifting all parameters by $1-\beta_n$ as in the complex case, according to the following lemma.  

\begin{lem}\label{lem:shift}
If $\qq\delta \in \Z$, then 
\[ H_q(\pmb{\alpha}+\delta, \pmb{\beta}+\delta, t) = \omega((-1)^nt)^{\delta} H_q(\pmb{\alpha}, \pmb{\beta}, t). \]
\end{lem}

\begin{proof}
Immediate from reordering the sum, replacing $\mu \leftarrow \mu + \qq\delta$.
\end{proof}

\begin{rmk} \label{rmk:whyperiodtwist}
\Cref{defn:ourtwisthgm} is called \emph{period-normalized} because it takes into account the factor $\gamma_{\pmb{\alpha},\pmb{\beta}}$ in Euler's integral formula (\Cref{lem:euler}), following the convention of Fuselier--Long--Ramakrishna--Swisher--Tu \cite{FLRST}.  It is a twist of the definition in Beukers--Cohen--Mellit \cite{BCM} by (the value of) a Hecke Grossencharacter over $\Q(\zeta_m)$ (\Cref{thm:heckechar}), more precisely we have
$$
H(\pmb{\alpha}, \pmb{\beta}, t) = G(\pmb{\alpha},\pmb{\beta}) H^{\text{BCM}}(\pmb{\alpha},\pmb{\beta}, t).
$$

\Cref{defn:ourtwisthgm} depends on pairing up the numerator and denominator parameters (retaining simultaneous permutation by $S_n$), whereas the definition given by Beukers--Cohen--Mellit \cite{BCM} allows permuting numerator and denominator parameters separately, by $S_n \times S_n$.  On the other hand, our definition has a simpler shift given in \Cref{lem:shift} and gives us the simplest theorem statement.  Both normalizations have their advantages and reflect the fact that one can always twist hypergeometric motives.  For a complete comparison with the many other versions of finite field hypergeometric functions given in the literature, we refer the reader to \cite[\S 4.4]{FLRST}.
\end{rmk}

Recall the canonical isomorphism $\Gal(\Q(\zeta_m)\,|\,\Q) \cong (\Z/m\Z)^\times$.  Let $H$ be the subgroup of $k \in (\Z/m\Z)^\times$ such that $k\pmb{\alpha},k\pmb{\beta}$ is a permutation of $\pmb{\alpha},\pmb{\beta}$ (simultaneous reordering), and let $K_{\pmb{\alpha},\pmb{\beta}} = \Q(\zeta_m)^H \subseteq \Q(\zeta_m) \subseteq \C^\times$ be the subfield fixed under $H$.  

\begin{lem} \label{lem:fofd}
We have $H_q(\pmb{\alpha},\pmb{\beta},t) \in K_{\pmb{\alpha},\pmb{\beta}}$ and $H_q(p\pmb{\alpha},p\pmb{\beta},t) = H_q(\pmb{\alpha},\pmb{\beta},t^p)$.
\end{lem}

\begin{proof}
Repeat the proof of \cite[Lemma 3.2.10]{DKSSVW20}, observing that $G(\pmb{\alpha},\pmb{\beta})$ has the same invariance so the conclusion holds as well for the period-normalized sum.
\end{proof}

\begin{defn}\label{def:nondegenerate}
Let $\pmb{\alpha}, \pmb{\beta}$ be hypergeometric parameters with $\#\pmb{\alpha} = \#\pmb{\beta}$. We say that $\alpha_i$ or $\beta_i$ is a \defi{degenerate parameter} if $\beta_i-\alpha_i \in \Z$.  We say the parameters are \defi{degenerate} if there exists a degenerate parameter, else they are \defi{nondegenerate}.  Finally, we say the parameters are \defi{isotypically degenerate} if there is exactly one degenerate parameter. 
\end{defn}

\begin{exm} \label{exm:isotypic}
The parameters $\pmb{\alpha}=(1/2,3/4)$, $\pmb{\beta}=(0,1/4)$ are nondegenerate; $\pmb{\alpha}=(1/2,1/2,1/4)$ and $\pmb{\beta}=(1/2,1/2,3/4)$ are isotypically degenerate with the unique degenerate parameter $1/2$ having multiplicity $2$, and $\pmb{\alpha}=(1/2,0,1/4)$ and $\pmb{\beta}=(1/2,0,3/4)$ are degenerate but not isotypically degenerate. 
\end{exm}

\begin{rmk}
The condition in \Cref{def:nondegenerate} has previously been known as \emph{primitive} in the literature \cite{FLRST}; however, the condition is not that a failure of coprimality of some kind, but rather that the period-normalized hypergeometric sum decomposes nontrivially as sums in the degenerate case when one considers the Gauss sums in light of \Cref{BCM Jacobi Prop}.  
\end{rmk}

\subsection{Hypersurface point counts}

We now prove an identity that allows us to use our period-normalized hypergeometric sum to count points when the hypergeometric parameters are nondegenerate. This is in analogy with the recursive definition given in \cite{FLRST}, but our proof does not require any recursive argument or definition.

For the remainder of this section, we suppress the dependence on $q$ and write just $H=H_q$.

\begin{prop}\label{main hypergeometric formula}
Let $\pmb{\alpha}, \pmb{\beta} \subset (1/\qq)\Z$ be nondegenerate parameters such that $\#\pmb{\alpha} = \#\pmb{\beta}=n$.  Let $t \in \F_q \smallsetminus \{0,1\}$.  Then
\begin{equation} \label{eqn:xidf}
 \sum_{\substack{(x_i) \in (\F_q \setminus\{0,1\})^n \\ tx_1 \cdots x_n = 1}}\left(\prod_{i=1}^{n} \omega^{\alpha_i}(-x_i) \omega^{\beta_i - \alpha_i}(1-x_i)\right) =-H(\pmb\alpha, \pmb\beta, t).
\end{equation}
\end{prop}

\begin{rmk}
     The summation in \eqref{eqn:xidf} does not change if we add in points where $x_i=0$ or $1$ for any $i$ since our convention in  \textsection\ref{sec:ffhg} imposes that $\chi(0)=0$ for all characters $\chi$.
\end{rmk} 

\begin{proof}
By \Cref{BCM Jacobi Prop} using the nondegeneracy hypothesis along with \Cref{def:jacobi}, 
\begin{equation} \label{eqngmpso}
\begin{aligned}
G(\pmb\alpha +\m/\qq, \pmb\beta + \m/\qq) &= 
\prod_{i=1}^n \frac{g( \alpha_i + \m/\qq)g(-\beta_i- \m/\qq)}{g(\alpha_i - \beta_i)} 
= \prod_{i=1}^n J(\alpha_i + \mu/\qq, -\mu/\qq-\beta_i)
\\
&= 
\prod_{i=1}^n \left( \sum_{x_i \in \F_q \setminus\{0,1\}} \omega^{\m/\qq + \alpha_i}(x_i) \omega^{-\m/\qq - \beta_i}(1-x_i)\right) \\
&= \prod_{i=1}^n \left(\sum_{x_i \in \F_q \setminus\{0,1\}} \omega^{\mu/\qq + \alpha_i}(-x_i) \omega^{\beta_i - \alpha_i}(1-x_i) \right),
\end{aligned}
\end{equation}
where the last line follows from applying \Cref{mobius trick}.

We then plug \eqref{eqngmpso} into \Cref{defn:ourtwisthgm} to obtain
\begin{equation}\begin{aligned}
H(\pmb \alpha, \pmb\beta, t) &= - \frac{1}{\qq} \sum_{\m = 0}^{\qq -1}\left(\prod_{i=1}^n \sum_{x_i \in \F_q \setminus\{0,1\}} \omega^{\alpha_i}(-x_i) \omega^{\beta_i-\alpha_i }(1-x_i)\omega^{\m/\qq}(-x_i) \right) \omega((-1)^n t)^{\m/\qq} \\
    &\qquad= -\sum_{(x_i) \in (\F_q \setminus\{0,1\})^n} \prod_{i=1}^n \omega^{\alpha_i}(-x_i) \omega^{\beta_i-\alpha_i }(1-x_i) \left( \frac{1}{\qq} \sum_{\m = 0}^{\qq -1}\omega(t x_1\cdots x_d)^{\m/\qq} \right).
\end{aligned}
\end{equation}
Finally, by character theory the latter term is $1$ or $0$ as $tx_1\cdots x_d=1$ or not; so the sum becomes
\begin{align*}
H(\pmb \alpha, \pmb\beta, t)
    &= -\sum_{\substack{(x_i) \in (\F_q \setminus\{0,1\})^n \\ tx_1 \cdots x_n = 1}} \prod_{i=1}^n \omega^{\alpha_i}(-x_i) \omega^{(\beta_i-\alpha_i) }(1-x_i). \qedhere
\end{align*}
\end{proof}

\subsection{Hypergeometric $L$-series}  \label{sec:hyplseries}

To conclude, we globalize.  Suppose that $\alpha_i-\beta_j \not\in\Z$ for all $i,j=1,\dots,n$ (a fundamental premise of our paper is to analyze degeneracies separately).  
Let $t \in \Q \smallsetminus \{0,1\}$, and let $S(\pmb{\alpha},\pmb{\beta}, t)$ be the set of primes dividing $m$ together with the numerator or denominator of either $t$ or $t-1$; we call these primes \defi{bad}.

Let $p \not\in S(\pmb{\alpha},\pmb{\beta})$ be a good prime, let $q=p^f=\Nm(\frakp)$ for any prime $\frakp$ in $\Z[\zeta_m]$ above $p$.  Then in particular $\qq \pmb{\alpha}, \qq\pmb{\beta} \subset \Z$.  We recall (\Cref{defn:ourtwisthgm}, \Cref{lem:fofd}) the period-normalized hypergeometric sum $H_q(\pmb{\alpha}, \pmb{\beta}, t) \in K_{\pmb{\alpha}, \pmb{\beta}} \subseteq \Q(\zeta_m)$.  We define its exponential generating series
\begin{equation}
    L_q(H(\pmb{\alpha}, \pmb{\beta}, t), T) \colonequals  \exp\left(\sum_{r=1}^\infty -H_{q^r}(\pmb{\alpha}, \pmb{\beta}, t)\frac{T^r}{r}\right) \in 1+ TK_{\pmb{\alpha}, \pmb{\beta}}[[T]].
\end{equation}
and the product
\begin{equation} \label{Lpyup}
\begin{aligned}
L_p(H(\pmb{\alpha},\pmb{\beta},t), \Q(\zeta_m), T) &\colonequals  \prod_{k \in (\Z/m\Z)^\times/\langle p \rangle} L_q(H(k\pmb{\alpha},k\pmb{\beta},t),T^f) \in 1+TK_{\pmb{\alpha},\pmb{\beta}}[[T]].
\end{aligned}
\end{equation}

\begin{lem} \label{lem:Hqabdef}
We have
\begin{equation} 
L_p(H(\pmb{\alpha},\pmb{\beta},t), \Q(\zeta_m), T) \in 1+T \Q[[T]]
\end{equation}
and 
\begin{equation} \label{eqn:LpHk}
L_p(H(\pmb{\alpha},\pmb{\beta},t), \Q(\zeta_m), T)=L_p(H(k\pmb{\alpha};k\pmb{\beta},t), \Q(\zeta_m), T)
\end{equation}
for all $k \in \Z$ coprime to $p$ and $m$.  
\end{lem}

\begin{proof}
Repeat the argument in \cite[Lemma 4.1.9]{DKSSVW20} (our sum is period-normalized, but the same argument applies).  
\end{proof}

The following result due to Katz is foundational.  

\begin{thm}[Katz] \label{Katzthm}
$L_q(H(\pmb{\alpha},\pmb{\beta},T) \in 1+T\Q(\zeta_m)[T]$ is a polynomial of degree $n$.  
\end{thm}

\begin{proof}
See \cite[Theorem 8.4.2]{Katz}.  The finite field hypergeometric sum \cite[(8.2.7)]{Katz} is the trace on a sheaf of rank $n$ obtained from convolution.  (More generally, see also Katz's work on rigid local systems \cite[Chapter 8]{Katz96}.)
\end{proof}

\section{Point counts in cyclic covers and for the family \texorpdfstring{$Y_{\aunder, \bunder, m}$}{Yabm}}

In this section, we count points on fibers of the family $X_{\aunder, \bunder, m}$ defined in \eqref{eqn:alleqns}, obtained as a partial compactification of the family defined in~\eqref{eqn:thefam}.  We continue the notation from the previous sections.

\subsection{Character sums and primitive point counts}

Throughout this section, let $q$ be a prime power with $\gcd(m,q)=1$, and let $m_q \colonequals \gcd(m,\qq)$.  To count points over $\F_q$, we require the following basic proposition which relates them to character sums. 

\begin{prop}\label{Branch point count}
For all $x \in \F_q^\times$,
\[
\# \{ y \in \F_q^\times : y^m = x\} = \# \{ y \in \F_q^\times : y^{m_q} = x\} = \sum_{k=1}^{m_q} \omega^{k/m_q}(x).
\]
\end{prop}

\begin{proof}
Straightforward character theory, e.g.\ generalizing Ireland--Rosen \cite[Proposition 8.1.5]{IR90}.
\end{proof}

In particular, \Cref{Branch point count} implies that our point counts below over $\F_q$ will only need to consider characters of order dividing $m_q$. 

We organize the sum in \Cref{Branch point count} into pieces as follows.  Suppose that $t \in \F_q \smallsetminus \{0,1\}$.  We define the \defi{primitive} point count for $Y_{\aunder,\bunder,m_q,t}$ over $\F_q$ by
\begin{equation} \label{Pabnt}
P_{\aunder,\bunder,m_q,t}(\F_q) \colonequals \sum_{k \in (\Z/m_q\Z)^\times} \sum_{\substack{x \in (\F_q\smallsetminus\{0,1\})^n \\ tx_1\cdots x_n=1}} \omega^{k/m_q}(f_{\aunder,\bunder}(x)).
\end{equation}

\begin{exm}
When $m_q=1$, we interpret the sum as over one index $k=0$, and
\[ P_{\aunder,\bunder,1,t}(\F_q) = \#U_n(\F_q) = \#\{x \in (\F_q \smallsetminus \{0,1\})^n : tx_1\cdots x_n=1\}. \]
We have 
\[ \#U_n(\F_q) = (q-2)^{n-1} - \#U_{n-1}(\F_q) \]
and $\#U_1(\F_q) = 1$ (since $t \neq 1$), so 
\[ P_{\aunder,\bunder,1,t}(\F_q) = \sum_{i=0}^{n-1} (-1)^i (q-2)^{n-1-i} = \frac{(q-2)^n - (-1)^n}{q-1}. \]
\end{exm}

\begin{lem} \label{slamlem}
 We have
\[ 
\#Y_{\aunder,\bunder,m,t}(\F_q) = \#Y_{\aunder,\bunder,m_q,t}(\F_q) = \sum_{d \mid m_q} P_{\aunder,\bunder,d,t}(\F_q). \]
\end{lem}

\begin{proof}
Applying \Cref{Branch point count} and organizing according to $d=\gcd(k,m_q)$, still with $m_q=\gcd(m,\qq)$ we have
\begin{equation}
\begin{aligned} 
\#Y_{\aunder,\bunder,m,t}(\F_q) &=  
\sum_{\substack{x \in (\F_q \smallsetminus \{0,1\})^n \\ tx_1\cdots x_n=1}} \sum_{k=1}^{m_q} \omega^{k/m_q}(f_{\aunder,\bunder}(x)) \\
&=  \sum_{d \mid m_q} \sum_{k \in (\Z/(m_q/d)\Z)^\times} 
\sum_{\substack{x \in (\F_q \smallsetminus \{0,1\})^n \\ tx_1\cdots x_n=1}} \omega^{dk/m_q}(f_{\aunder,\bunder}(x)) 
=\sum_{d \mid m_q} P_{\aunder,\bunder,m_q/d,t}(\F_q) 
\end{aligned}
\end{equation}
which gives the result swapping $d$ and $m_q/d$.
\end{proof}

Recall $Y'_{\aunder, \bunder, m, I,t}$ from ~\eqref{def: Y a b m I subset}. Since $Y'_{\aunder, \bunder, m, I,t}$ is a twist of $Y_{\aunder_I, \bunder_I, m, t}$, defining
\begin{equation}\label{P' abdIt}
P'_{\aunder,\bunder,m,I,t}(\F_q) \colonequals \sum_{k \in (\Z/m_q\Z)^\times} (-1)^{(k\qq/m)(\sum_{i \not\in I} a_i)} \sum_{\substack{x \in (\F_q\smallsetminus\{0,1\})^{\#I} \\ t\prod_{i \in I} x_i=1}} \omega^{k/m_q}(f_{\aunder_I,\bunder_I}(x)),
\end{equation}
in much the same way we conclude the following.

\begin{cor}\label{slamlem With Twist}
We have
\[ 
\#Y'_{\aunder,\bunder,m,I,t}(\F_q) = \#Y'_{\aunder,\bunder,m_q,I,t}(\F_q) = \sum_{d \mid m_q} P'_{\aunder,\bunder,d,I,t}(\F_q). \]
\end{cor}
\begin{proof}
    The proof is almost verbatim that of \Cref{slamlem}, after taking into consideration the twist.\end{proof}

\begin{rmk}
We use the word \emph{primitive} because we are motivated by thinking about these point counts as arising from the action of Frobenius on compactly supported \'etale cohomology; then the group scheme $\mu_m$ acts, and we think of the above pieces as corresponding to those where this group acts by a primitive $d$th root of unity for $d \mid m_q$.  Unfortunately, we are not able to make this analogy a method of proof, because we do not have the necessary control over this cohomological action---so we must be indirect instead.  
\end{rmk}

\subsection{Setup}

By definition, the scheme $X_{\aunder,\bunder,m}$ is naturally stratified based on whether or not $x_i=1$.  We recall \eqref{def: Y a b m I subset}, that for $I \subseteq \{1,\dots,n\}$, the restriction of $X_{\aunder,\bunder,m}$ to the subset where $x_i \neq 1$ for $i \in I$ and $x_i=1$ for $i \not\in I$ is given by $Y'_{\aunder,\bunder,m,I}$ by \Cref{lem:xi1}.  Hence
\begin{equation} \label{eqn:xunder}
\#X_{\aunder,\bunder,m,t}(\F_q) = \sum_{I \subseteq \{1,\dots,n\}} 
\#Y'_{\aunder, \bunder, m, I,t}(\F_q) = \sum_{I \subseteq \{1,\dots,n\}} 
\#Y'_{\aunder, \bunder, m_q, I,t}(\F_q).
\end{equation}
We remark and emphasize that this is exactly possible due to the construction of the partial compactification.  In light of \eqref{eqn:xunder}, without loss of generality we may suppose that $m=m_q=\gcd(m,\qq)$.  

Next, we decompose the point counts on the open sets into their primitive parts using \Cref{slamlem With Twist}. By substituting into \eqref{eqn:xunder}, recalling $\aunder_I,\bunder_I,m_I,$ in~\eqref{aIbImI}, and interchanging the order of summation, we obtain:
\begin{equation}\begin{aligned} \label{interchangedm}
\sum_{I} \#Y'_{\aunder, \bunder, m, I,t}(\F_q) &= 
\sum_{I} \sum_{d \mid m_I} \sum_{ k \in (\Z/d\Z)^{\times}} \sum_{\substack{(x_i) \in (\F_q \setminus\{0,1\})^{I} \\ t \prod_{i \in I}x_i=1}}  \omega^{k \sum_{i \notin I} a_i/d}(-1) \omega^{k/d}(f_{\aunder_I,\bunder_I,m_I}(x)) \\
&= 
\sum_{d \mid m} \sum_{\substack{I \\ d \mid m_I}} \sum_{ k \in (\Z/d\Z)^{\times}} \sum_{\substack{(x_i) \in (\F_q \setminus\{0,1\})^{I}\\ t \prod_{i \in I}x_i=1}}  \omega^{k \sum_{i \notin I} a_i/d}(-1) \omega^{k/d}(f_{\aunder_I,\bunder_I,m_I}(x))\\
&= \sum_{d \mid m} \sum_{\substack{I \\ d \mid m_I}} P'_{\aunder,\bunder,d,I,t}(\F_q).
\end{aligned}\end{equation}

To proceed, we need to study the fibers of the map $I \mapsto m_I$.  

\begin{lem}\label{lem:poset structure}
The map
\[ 
\begin{aligned}
\mathcal{P}(\{1,\dots,n\}) &\to \{d : d \mid m\}  \\
I &\mapsto m_I = \gcd(\{b_i-a_i: i \not\in I\} \cup \{m\})
\end{aligned} \]
from the power set to the set of divisors of $m$ is a map of posets: if $I \subseteq I'$ then $m_I \mid m_{I'}$.  
\end{lem}

\begin{proof}
This is straightforward upon unravelling the partial orderings.  
\end{proof}

To organize the subsets $I$ as in \eqref{interchangedm}, it is natural to look at the subset of indices $i$ such that $d \mid (b_i-a_i)$; in terms of posets, by definition this set is the maximal $I$ such that $d \mid m_I$. However, it is more convenient to work with the complement (so a minimal set).  We define
\begin{equation} \label{eqn:Id}
I_d \colonequals \{ i : d \nmid (b_i-a_i) \} \subseteq [n] \colonequals \{1,\dots,n\}.
\end{equation}
We thus can write, using \Cref{lem:poset structure},
\begin{equation} \label{get it down to P'd}
\sum_{I} \#Y'_{\aunder, \bunder, m, I,t}(\F_q) = \sum_{d \mid m} \sum_{I \supseteq I_d}  P'_{\aunder,\bunder,d,I,t}(\F_q).
\end{equation}

\subsection{Explicit hypergeometric point count}

We now observe a remarkable M\"obius-like cancellation; this theorem forms the technical heart of the paper.  

\begin{thm}\label{hypergeometric main}
Let 
$$
Q'_{d,t} = Q'_{\aunder,\bunder,d,t} \colonequals \sum_{I \supseteq I_d}  P'_{\aunder,\bunder,d,I,t}(\F_q).
$$
\begin{enumalph}
\item If $[n] = I_d$ then 
$$
Q'_{d,t} = - \sum_{k \in (\Z/d\Z)^\times} H_q(k\aunder/d,k\bunder/d,t).
$$
\item Suppose $[n] \ne I_d$ and $a_i - a_j \in d\Z$ for all $i,j \in [n]\setminus I_d$.  Let $j$ such that $j \notin [n]\setminus I_d$, and let $c=a_j$.  Then
\begin{equation} \label{eqn:pdisotyp}
\begin{aligned}
Q'_{d,t} &= (-1)^{\#I_d} \sum_{ k \in (\Z/d\Z)^{\times}}  \omega^{-kc / d}(t) (\qq)^{n-\#I_d-1} \\
	&\qquad \cdot \Biggl( \prod_{\substack{ i \in I_d\\c - b_i \notin d\Z}} \psi_{\tfrac{k(a_i - c)}{d}, \tfrac{k(b_i - a_i)}{d}} \Biggr) \Biggl(\omega^{k\sum_{i=1}^n a_i / d}(-1) \prod_{\substack{ i \in I_d\\ c - b_i \in d\Z}} \omega^{k(a_i-c)/d}(-1) \Biggr).
\end{aligned}
\end{equation}
\item If $[n] \ne I_d$ and there exists $i,j \in [n]\setminus I_d$ so that $a_i - a_j \notin d\Z$, then $Q'_{d,t} = 0$.
\end{enumalph}
\end{thm}

\begin{proof}
For (a), we start with the definition \eqref{P' abdIt} and use the hypothesis to simplify to
\[  \sum_{I \supseteq I_d}  P'_{\aunder,\bunder,d,I,t}
    = \sum_{k \in (\Z/d\Z)^\times} \sum_{\substack{x \in (\F_q \smallsetminus \{0,1\})^n \\
    tx_1\cdots x_n=1}} \prod_{i=1}^{n} \omega^{ka_i/d}(-x_i) \omega^{k(b_i-a_i)/d}(1-x_i). \]
Applying \Cref{main hypergeometric formula} obtains the result.

For (b) and (c), choose some $j \in [n] \setminus I_d$, i.e., $b_j - a_j \in d \Z$ and write $I_j := I\cup\{j\}$.  Then
\begin{equation} \label{532} \begin{aligned}
 Q_{d,t}' &=  \sum_{ \substack{I \\ I_d \subseteq I \subseteq [n] \setminus\{j\} }}  P'_{\aunder,\bunder,d,I,t}. +  P'_{\aunder,\bunder,d,I_j,t}.\\
	&=  \sum_{I} \sum_{ k \in (\Z/d\Z)^{\times}}  \prod_{i \notin I} \omega^{ka_i/d}(-1) \left(\sum_{\substack{(x_i) \in (\F_q \setminus\{0,1\})^{I}\\ t \prod_{i \in I}x_i=1}} \prod_{i \in I} \omega^{k a_i /d}(-x_i) \omega^{k(b_i - a_i) / d}(1-x_i)\right.  \\
	& \qquad \left. \qquad\qquad\qquad+ \sum_{\substack{(x_i) \in (\F_q \setminus\{0,1\})^{I_j} \\ t \prod_{i \in I_j }x_i=1}}  \omega^{ka_j/d}(x_j)  \prod_{i \in I} \omega^{k a_i /d}(-x_i) \omega^{k(b_i - a_i) / d}(1-x_i)\right) \\
	&=  \sum_{I} \sum_{k}   \omega^{k\sum_{i\notin I}a_i/d}(-1) \sum_{\substack{(x_i) \in (\F_q \setminus\{0,1\})^{I} \\ x_j \in \F_q \setminus\{0\}\\ t \prod_{i \in I_j }x_i=1}} \omega^{ka_j/d}(x_j)  \prod_{i \in I} \omega^{k a_i /d}(-x_i) \omega^{k(b_i - a_i) / d}(1-x_i).
\end{aligned}\end{equation}
The first equality is from partitioning the summation over when $I$ contains $j$ or not, the second is pulling out a common character, and the third is realizing one can combine the sums into a summation over all $x_j\in \F_q$ since the first sum is when $x_j = 1$. We now substitute $ x_j = (t \prod_{i \in I}x_i)^{-1}$ and continue.

For each $I \supseteq I_d$ and $k \in (\Z/d\Z)^\times$, we have
\begin{equation} \label{533} \begin{aligned}
 	&\prod_{i \notin I} \omega^{ka_i/d}(-1) \sum_{(x_i) \in (\F_q \setminus\{0,1\})^{I}} \omega^{-ka_j/d}(t \textstyle{\prod}_{i \in I} x_i)  \prod_{i \in I} \omega^{k a_i /d}(-x_i) \omega^{k(b_i - a_i) / d}(1-x_i) \\
	&\qquad =  \omega^{k(\sum_{i=1}^n a_i)/d}(-1)\omega^{-ka_j / d}(t) \sum_{(x_i) \in (\F_q \setminus\{0,1\})^{I}} \prod_{i \in I} \omega^{k (a_i-a_j) /d}(-x_i) \omega^{k(b_i - a_i) / d}(1-x_i)\\
	&\qquad =   \omega^{k(\sum_{i=1}^n a_i)/d}(-1)\omega^{-ka_j / d}(t)  \prod_{i \in I} \sum_{x_i \in (\F_q \setminus\{0,1\})}\omega^{k (a_i-a_j) /d}(-x_i) \omega^{k(b_i - a_i) / d}(1-x_i)\\
	&\qquad =  \omega^{k(\sum_{i=1}^n a_i)/d}(-1)\omega^{-ka_j / d}(t)  \prod_{i \in I} J(k(a_i - a_j)/d, k(b_i - a_i) /d).
\end{aligned}\end{equation}

We focus on the quantity $J(k(a_i - a_j)/d, k(b_i - a_i) /d)$. This quantity can be made explicit. If $i \in I_d$, then we apply \Cref{cor:case by case Jacobi} when $a_j - b_i \notin d\Z$ and \Cref{rest for degen case} when $a_j - b_i \in d\Z$ to obtain  \begin{equation}
J(k(a_i - a_j)/d, k(b_i - a_i) /d) =\begin{cases} \psi_{\tfrac{k(a_i - a_j)}{d}, \tfrac{k(b_i - a_i)}{d}}, & \text{ if $b_i - a_j \notin d\Z$}; \\ -\omega^{k(a_i-a_j)/d}(-1), & \text{ if $b_i - a_j \in d\Z$}.\end{cases}
\end{equation}
If $i \notin I_d$ then $b_i - a_i \in d\Z$, so
\begin{equation}
J(k(a_i - a_j)/d, k(b_i - a_i) /d) = \sum_{x \in \F_q \setminus\{0,1\}} \omega^{k(a_i - a_j)/d}(x) = \begin{cases} -1 & \text{if $a_i - a_j \notin d\Z$}; \\ \qq -1 & \text{ if $a_i - a_j \in d\Z$}.\end{cases} 
\end{equation}
Substituting these into \eqref{533}, summing over $I$ and $k$, and substituting back into \eqref{532} we obtain
\begin{equation}\begin{aligned}
 Q_{d,t}' &= \sum_{I \supseteq I_d} \sum_{ k \in (\Z/d\Z)^{\times}} \sum_{\substack{(x_i) \in (\F_q \setminus\{0,1\})^{I}\\ t \prod_{i \in I}x_i=1}}  \omega^{k \sum_{i \notin I} a_i/d}(-1) \omega^{k/d}(f_{\aunder_I,\bunder_I,m_I}) \\
 	&=  \sum_{ k \in (\Z/d\Z)^{\times}}  \omega^{k\sum_{i=1}^n a_i / d}(-1)\omega^{-ka_j / d}(t) \left( \prod_{\substack{ i \in I_d\\ a_j - b_i \notin d\Z}} \psi_{\tfrac{k(a_i - a_j)}{d}, \tfrac{k(b_i - a_i)}{d}} \right)  \\
	&\qquad \left(\prod_{\substack{ i \in I_d\\ a_j - b_i \in d\Z}} -\omega^{k(a_i-a_j)/d}(-1) \right)\cdot \sum_{ \substack{I\\I_d \subseteq I \subseteq [n] \setminus\{j\} }} \left(\prod_{\substack{ i \in I \setminus I_d\\a_i - a_j \notin d\Z} } -1 \right) \left( \prod_{\substack{ i \in I \setminus I_d \\a_i - a_j \in d\Z} } \qq-1 \right) \end{aligned}\end{equation}
We focus on this second line. We can compute by rearranging the sum into a product of binomials
\begin{equation}
 \sum_{ \substack{I \\I_d \subseteq I \subseteq [n] \setminus\{j\} }} \left(\prod_{\substack{ i \in I \setminus I_d \\a_i - a_j \notin d\Z} } (-1) \right) \left( \prod_{\substack{ i \in I \setminus I_d \\ a_i - a_j \in d\Z} } (\qq-1) \right)  = \left(\prod_{\substack{ i \in I \setminus I_d \\ a_i - a_j \notin d\Z} } ((-1) + 1) \right)\left( \prod_{\substack{ i \in I \setminus I_d \\ a_i - a_j \in d\Z} } ((\qq-1) + 1) \right).
\end{equation}
So the entire quantity vanishes if there exists $j,k \in [n] \setminus I_d$ so that $a_k - a_j \notin d \Z$, proving (c). Otherwise, if $ a_k - a_j \in d\Z$ for all $j, k \in [n] \setminus I_d$, we obtain 
\begin{equation} \begin{aligned}
 Q_{d,t}' &= \sum_{I \supseteq I_d}\sum_{ k \in (\Z/d\Z)^{\times}} \sum_{\substack{(x_i) \in (\F_q \setminus\{0,1\})^{I}\\ t \prod_{i \in I}x_i=1}}  \omega^{k \sum_{i \notin I} a_i/d}(-1) \omega^{k/d}(f_{\aunder_I,\bunder_I,m_I}) \\
 	&= \sum_{ k \in (\Z/d\Z)^{\times}}  \omega^{k\sum_{i=1}^n a_i / d}(-1)\omega^{-ka_j / d}(t) (\qq)^{n-\#I_d-1} \\
	&\qquad \cdot \Biggl( \prod_{\substack{ i \in I_d\\ a_j - b_i \notin d\Z}} \psi_{\tfrac{k(a_i - a_j)}{d}, \tfrac{k(b_i - a_i)}{d}} \Biggr) \Biggl(\prod_{\substack{ i \in I_d\\ a_j - b_i \in d\Z}} -\omega^{k(a_i-a_j)/d}(-1) \Biggr) \end{aligned}\end{equation}
as claimed.
\end{proof}

\begin{rmk}\label{rmk: upshot of thm}
We are led to think of the outcome of \Cref{hypergeometric main} as some indication that in a degenerate setting, the hypergeometric motive becomes a Jacobi motive.  The expression in \Cref{hypergeometric main}(b) is complicated because we work over the torus, leaving pieces in every dimension.  We leave the further pursuit of this notion of degeneration for future work.  
\end{rmk}

\subsection{Proof of main result}

We now are now staged to complete the proof of our main result, taking the point counts over finite fields and assembling them into the zeta function.  

Recalling \textsection\ref{sec:hyplseries}, let $S = S(\pmb{\alpha},\pmb{\beta},t)$ be the set of bad primes.  Then by definition
\begin{equation}  \label{eqn:zetasP}
\zeta_S(X_{\aunder,\bunder,m,t},s) = \prod_{p \not \in S} Z_p(X_{\aunder,\bunder,m,t},p^{-s}) 
\end{equation}
where $Z_p(X_{\aunder,\bunder,m,t},T) \in \Q(T)$ is the zeta function (exponential generating series) over $\F_p$.  

We recall (\Cref{thm:heckechar}) that Gauss sums give rise to Hecke Grossencharacters for $\Q(\zeta_m)$.  Anticipating \eqref{eqn:pdisotyp} (in \Cref{hypergeometric main}(b)), we will need a twist, defined as follows.  Let $d \mid m$, and suppose that $\aunder/d,\bunder/d$ is isotypically degenerate for $c/d$. As before, we write  $e(\aunder/d,\bunder/d)$ for the multiplicity of $c/d$ in $\aunder/d$. Then for $\frakp \subseteq \Z[\zeta_d]$ above a good prime $p$, 
we define 
\begin{equation} \label{eqn:explicithypergross}
\psi_{\aunder,\bunder,d,t}(\frakp) = \omega^{-c / d}(t) \Biggl(\omega^{\sum_{i=1}^n a_i / d}(-1) \prod_{\substack{ i \in I_d\\ c - b_i \in d\Z}} \omega^{k(a_i-c)/d}(-1) \Biggr) \Biggl( \prod_{\substack{ i \in I_d\\c - b_i \notin d\Z}} \psi_{\tfrac{a_i - c}{d}, \tfrac{b_i - a_i}{d}}(\frakp) \Biggr)
\end{equation}
where we recall \eqref{eqn:Id}
\[ I_d = \{ i : d \nmid (b_i-a_i) \} \subseteq [n] \colonequals \{1,\dots,n\}. \]

\begin{thm} \label{thm:mainthm-insitu}
For $p$ good, we have
\[ Z_p(X_{\aunder,\bunder,m,t},T) = \prod_{d \mid m} 
\begin{cases}
L_p(H(\aunder/d,\bunder/d,t),\Q(\zeta_d),T), & \text{if $\aunder/d,\bunder/d$ is nondegenerate;} \\
Z_p((\Gm)^{e(\aunder/d,\bunder/d)-1}, \Q(\zeta_d), T, \psi_{\aunder,\bunder,d,t})^{(-1)^{n-e(\aunder/d,\bunder/d)}}, & \begin{aligned}
    &\parbox[t]{0.35\linewidth}{if $\aunder/d,\bunder/d$ is isotypically \\ \phantom{yup} degenerate;}
  \end{aligned} \\
1, & \text{else}.
\end{cases} \]
\end{thm}

\begin{proof}
Let $q=p^r$ with $r \geq 1$.  We get organized by \eqref{eqn:xunder} and note that after that we need to substitute $m_q=\gcd(m,\qq)$ for $m$.  Then plug in \eqref{get it down to P'd} to get
\begin{equation} 
\#X_{\aunder,\bunder,m,t}(\F_q) = \sum_{d \mid m_q} \sum_{I \supseteq I_d}  P'_{\aunder,\bunder,d,I,t}(\F_q) = \sum_{d \mid m_q} Q'_{d,t}(\F_q),
\end{equation}
where the latter suppresses notation, staging for \Cref{hypergeometric main}.  Thus
\begin{equation} \label{eqn:expzp}
\begin{aligned}
\log Z_p(X,T) &= \sum_{r=1}^{\infty} \#X(\F_{p^r})\frac{T^r}{r} 
=\sum_{r=1}^{\infty} \biggl(\sum_{d \mid \gcd(m,p^r-1)} Q_{d,t}'(\F_{p^r}) \biggr)\frac{T^r}{r} \\
&=\sum_{d \mid m} \sum_{\substack{r \\ d \mid (p^r-1)}} Q_{d,t}'(\F_{p^r}) \frac{T^r}{r}.
\end{aligned}
\end{equation}
Now let $f_d$ be the order of $p$ modulo $d$, and let $q_d \colonequals p^{f_d}$; then $d \mid (p^r-1)$ if and only if $f_d \mid r$.  Exponentiating \eqref{eqn:expzp} and substituting then gives
\begin{equation}
\begin{aligned}
Z_p(X,T) &=\prod_{d \mid m} \exp\biggl( \sum_{r = 1}^{\infty} Q_{d,t}'(\F_{q_d^r}) \frac{(T^{f_d})^r}{r} \biggr).
\end{aligned}
\end{equation}
To finish, we plug in \Cref{hypergeometric main} in three cases.  In the first case, we recognize the sum over Galois conjugates in part (a) from \eqref{Lpyup}.  In the second case, we apply part (b) having defined the Hecke Grossencharacter in \eqref{eqn:explicithypergross}, again with its Galois conjugates.  The third case gives no contribution.
\end{proof}

\begin{proof}[Proof of \textup{\Cref{thm:mainthm}}]
For the partial compactification, see \Cref{prop:openimmers}; then in each term of the product in \eqref{eqn:zetasP} apply \Cref{thm:mainthm-insitu}.
\end{proof}

\section{Examples}

To conclude, we give a few examples.  Our computations are performed in \textsf{Magma} \cite{Magma}.

\subsection{Curve}

Consider the case $m = 4$ with $\pmb{a} = (1,0)$ and $\pmb{b} = (3,2)$. Then $Y$ is defined by 
$$
y^4 = (-x_1)(1-x_1)^2(1-x_2)^2; \qquad tx_1x_2 =1
$$
where $x_1, x_2 \ne 0,1$. Our partial compactification $X \subseteq (\Gm)^2 \times \A^3$ is defined by the equations
\[ tx_1x_2 =y_1=1, \qquad x_1,x_2 \neq 0, \] 
together with
\begin{align*}
    y_4^4 = (-x_1)(1-x_1)^2(1-x_2)^2; \quad y_4^2 = y_2(1-x_1)(1-x_2); \quad y_2^2 = (-x_1).
\end{align*}
\Cref{thm:mainthm} gives the zeta function of $X$ as the product over $d=1,2,4$.  We compute that $I_4 = \{1,2\}$, $I_2 = \emptyset$, and $I_1 = \emptyset$.  The first case has nondegenerate parameters; the second degenerate but not isotypic parameters $(1/2,0),(1/2,0)$; the third degenerate and isotypic parameters, and trivial Hecke Grossencharacter.  Thus 
\begin{equation} 
\zeta_S(X,s) = L(H(1/4,0;1/2,1/2),\Q(i),s) \zeta_S(\Gm,s). 
\end{equation}

Running through the proof a bit, with notation as in \Cref{hypergeometric main}, we have that
\begin{equation}
\#X_{\aunder,\bunder,m,t}(\F_q) =  \sum_{d \mid m} Q'_{\aunder,\bunder,d,I,t}
\end{equation}
where
\[     Q'_4 = -H(\tfrac14, 0; \tfrac12, \tfrac12;t) - H(\tfrac34, 0; \tfrac12, \tfrac12;t) \]
and $Q'_2=0$ and $Q'_1=\qq$, using each case exactly once.  We then obtain that 
$$
\#X_{\aunder,\bunder,m,t}(\F_q) = \begin{cases}\qq -H(\tfrac14, 0; \tfrac12, \tfrac12;t) - H(\tfrac34, 0;\tfrac12, \tfrac12;t), &\text{ if $q \equiv 1 \psmod 4$};\\
\qq, & \text{ otherwise}.\end{cases} 
$$

Through the zigzag procedure, one can see that the Hodge vectors for the two hypergeometric series are $(1,1)+(1,1)=(2,2)$.  


\subsection{Surface}
Consider the case $m=12$ with $\aunder=(1,3,6)$ and $\bunder-\aunder=(2,4,12)$ so $\bunder=(3,7,18)$.  Then $Y$ is defined by
\begin{align*}
y^{12} &= (-x_1)(1-x_1)^2 (-x_2)^3 (1-x_2)^4 (-x_3)^6 (1-x_3)^{12} \\
&= x_1(1-x_1)^2 x_2^3 (1-x_2)^4 x_3^6 (1-x_3)^{12}
\end{align*}
together with $tx_1 \cdots x_d=1$ and $x_1,x_2,\dots,x_d \neq 0,1$. 

Our partial compactification $X\subseteq (\Gm)^3 \times \A^6$ is defined by the equations 
\[ tx_1x_2 x_3=y_1=1, \qquad  x_1,x_2,x_3 \neq 0, \] 
together with
\begin{equation}\begin{aligned}\label{eqn:exmthreefold}
   y_{12}^{12} &=  (-x_1)(1-x_1)^2 (-x_2)^3 (1-x_2)^4 (-x_3)^6, 
   & y_{12}^{6} &= y_2  (1-x_1) (1-x_2)^2 (-x_3)^3,\\
   y_{12}^{4} &= y_3  (-x_2)  (-x_3)^2, 
   & y_{12}^{3} &= y_4  (1-x_2),\\
   y_{12}^{2} &= y_6 (-x_3),
   &y_{6}^{6} &=  (-x_1)(1-x_1)^2 (-x_2)^3(1-x_2)^4, \\
   y_{6}^{3} &= y_2  (1-x_1)(1-x_2)^2, 
   &y_{6}^{2} &= y_3  (-x_2), \\
   y_4^4 &= (-x_1)(1-x_1)^2(-x_2)^3(-x_3)^6,
   &y_4^2 &= y_2(1-x_1)(-x_3)^3,\\
   y_3^3 &=  (-x_1)(1-x_1)^2(1-x_2)^4,
   &y_2^2 &=  (-x_1)(-x_2)^3
\end{aligned}   
\end{equation}
When $x_1 =1$, the equations~\eqref{eqn:exmthreefold} above reduce to
$$
    y_2^2 = (-x_2)^3; \qquad y_3 = y_4 = y_6 = y_{12} = 0.
$$
When $x_2,x_3 \neq 1$, this is $Y_{(3,6),(7,18),2}\simeq Y_{(1,0),(1,0),2}$.  When $x_2 = 1$ and $x_3 \neq 1$, we get $y_2^2=-1$ in the variables $y_2,x_3$. In this case, we can see that the sign is twisted and see that this is isomorphic to $Y_{(0),(0),2}'$. If instead $x_3=1$, we get $Y_{(1),(1),2}$.

Next, examine the case where $x_2 =1$ and $x_1 \ne 1$. The equations~\eqref{eqn:exmthreefold} reduce to 
\begin{align*}
y_{12} &= y_6=y_3 = 0, &y_4^4 &= (-x_1)(1-x_1)^2(-x_3)^6,\\
y_4^2&=y_2(1-x_1)(-x_3)^3 &y_2^2&=(-x_1)
\end{align*}
Using that $x_3 \ne 0$, one can solve for and eliminate $y_2$, hence we obtain that this stratum is isomorphic to $Y_{(1,6),(3,6); 4} \simeq Y_{(1,2),(3,2);4}$ when $x_3 \ne 1$. If $x_3=1$, then it is isomorphic to $Y_{(1), (3); 4}$.

Lastly, when $x_3=1$ and $x_1, x_2 \ne 1$, we can solve for and eliminate $y_2, y_3, y_4, y_6$.  Substitute $x_3=1$ into the first equation in~\eqref{eqn:exmthreefold}, to see this stratum is isomorphic to $Y_{(1,3), (3, 7), 12}$.

We compute that $I_{12} =I_6=I_3 =\{1,2\}$, $I_4 = \{1\}$, and $I_2=I_1 = \emptyset$. By \Cref{hypergeometric main}, we compute $Q'_{\aunder,\bunder,d,I,t}$ for all $d \mid 12$ to be
\begin{align*}
      Q'_{\aunder,\bunder,12,I,t} &= \sum_{k \in (\Z/12\Z)^\times} \omega^{-k/2}(t)\psi_{(5k/12,3k/4),(k/6,k/3)}; & Q'_{\aunder,\bunder,4,I,t} &= 0; \\
    Q'_{\aunder,\bunder,6,I,t} &= \sum_{k \in (\Z/6\Z)^\times} \psi_{(k/6,k/2),(k/3,2k/3)}; & Q'_{\aunder,\bunder,2,I,t} &=0; \\
    Q'_{\aunder,\bunder,3,I,t} &= \sum_{k \in (\Z/3\Z)^\times} -\psi_{0,k/3} \omega^{2k/3}(-1) = -2;  &  Q'_{\aunder,\bunder,1,I,t} &= (\qq)^2.
\end{align*}



In this case, there is no hypergeometric motive arising at all!  For $d=12$, we get the Hodge numbers $(1,2,1)$ over $\Q(\zeta_{12})$ (so totalling $(4,8,4)$); for $d=6$, we get $(1,0,1)$ over $\Q(\zeta_6)$ (so totalling $(2,0,2)$). Lastly, for $d=1,3$, we get a contribution in the second cohomology of type $(0,2,0)$. In sum, we get a middle Hodge structure of type $(6,10,6)$.

\end{document}